\documentclass[11pt, reqno]{amsart}

\usepackage[utf8]{inputenc}
\usepackage[OT2,T1]{fontenc}
\usepackage[english]{babel}
\usepackage{lmodern}

\usepackage{amsmath}
\usepackage{amsfonts}
\usepackage{amssymb}
\usepackage{amsthm}
\usepackage{mathrsfs}

\usepackage[fleqn]{nccmath}

\usepackage[OT2,T1]{fontenc}
\DeclareSymbolFont{cyrletters}{OT2}{wncyr}{m}{n}
\DeclareMathSymbol{\Sha}{\mathalpha}{cyrletters}{"58}

\usepackage{url}
\usepackage[hidelinks]{hyperref}
\hypersetup{colorlinks=true,urlcolor=blue,citecolor=blue,linkcolor=blue}

\usepackage{multirow}

\usepackage[usenames,dvipsnames]{xcolor}
\usepackage{graphicx}
\usepackage{tikz-cd}

\usepackage{etoolbox}

\usepackage{colonequals}

\usepackage{enumerate}

\numberwithin{equation}{subsection}

\renewcommand{\theenumi}{\roman{enumi}}

\theoremstyle{plain}
\newtheorem{theorem}[equation]{Theorem}

\newtheorem{proposition}[equation]{Proposition}
\newtheorem{lemma}[equation]{Lemma}
\newtheorem{corollary}[equation]{Corollary}

\newtheorem{algorithm}[equation]{Algorithm}

\newtheorem*{MT1}{Main Result 1}
\newtheorem*{MT2}{Main Result 2}
\newtheorem*{MT3}{Main Result 3}

\theoremstyle{definition}
\newtheorem{definition}[equation]{Definition}

\theoremstyle{remark}
\newtheorem{remark}[equation]{Remark}
\newtheorem{example}[equation]{Example}

\def\defi{\textsf}
\def\ext{\!\mid\!}

\def\epsilon{\varepsilon}
\def\rho{\varrho}
\def\theta{\vartheta}
\def\phi{\varphi}
\def\tilde{\widetilde}

\def\Magma{\textsc{Magma}}

\def\loccit{\textit{loc.\ cit.}}

\DeclareMathOperator{\Aut}{Aut}

\DeclareMathOperator{\Cl}{Cl}

\DeclareMathOperator{\End}{End}

\DeclareMathOperator{\Gal}{Gal}
\DeclareMathOperator{\GL}{GL}

\DeclareMathOperator{\Hom}{Hom}
\DeclareMathOperator{\hyp}{hyp}
\DeclareMathOperator{\im}{im}
\DeclareMathOperator{\id}{id}

\DeclareMathOperator{\Jac}{Jac}

\DeclareMathOperator{\sgn}{sgn}

\DeclareMathOperator{\Tr}{Tr}

\def\a{\mathfrak{a}}
\def\b{\mathfrak{b}}
\def\aa{\mathfrak{a}}
\def\bb{\mathfrak{b}}
\def\cc{\mathfrak{c}}
\def\pp{\mathfrak{p}}

\def\p{\mathfrak{p}}

\def\C{\mathbb{C}}

\def\F{\mathbb{F}}

\def\P{\mathbb{P}}
\def\Q{\mathbb{Q}}
\def\R{\mathbb{R}}

\def\Z{\mathbb{Z}}

\def\aabar{\overline{\aa}}
\def\bbbar{\overline{\bb}}
\def\ccbar{\overline{\cc}}

\def\xbar{\overline{x}}

\def\Qbar{\overline{\Q}}

\def\CC{\mathbb{C}}

\def\ZZ{\mathbb{Z}}

\def\Cc{\mathcal{C}}
\def\Dc{\mathcal{D}}

\def\Mc{\mathcal{M}}
\def\Nc{\mathcal{N}}

\def\Df{\mathfrak{D}}

\def\Idf{\mathfrak{Id}}

\newcommand{\into}{\hookrightarrow}

\usepackage{fullpage}
\usepackage{mathtools}

\begin{document}

\title{Isogenous hyperelliptic and non-hyperelliptic Jacobians with maximal complex multiplication}

\date{\today}

\begin{abstract}
  We analyze complex multiplication for Jacobians of curves of genus 3, as well as the resulting Shimura class groups and their subgroups corresponding to Galois conjugation over the reflex field. We combine our results with numerical methods to find CM fields $K$ for which there exist both hyperelliptic and non-hyperelliptic curves whose Jacobian has complex multiplication by $\Z_K$. More precisely, we find all sextic CM fields $K$ in the LMFDB for which (heuristically) Jacobians of both types with CM by $\Z_K$ exist. There turn out to be 14 such fields among the 547,156 sextic CM fields that the LMFDB contains. We determine invariants of the corresponding curves, and in the simplest case we also give an explicit defining equation.
\end{abstract}

\thanks{The first and second author acknowledge financial support from the Thomas Jefferson Fund program, which is supported by the Embassy of France in the United States and the FACE Foundation. The third author was supported by the Juniorprofessuren-Programm ``Endomorphismen algebraischer Kurven'' of the Science Ministry of Baden-Württemberg as well as by a Sachbeihilfe ``Abstieg algebraischer
Kurven'' of the Deutsche Forschungsgemeinschaft. We thank John Boxall and Anna Somoza for bringing the André--Oort conjecture to our attention, as well as Marco Streng and Andrew Sutherland for their helpful comments on a previous version of this article and Tim Evink for his help with proving Proposition \ref{prop:torsor}. We thankfully acknowledge the use of computational resources provided by Dartmouth College. Finally, we are grateful to the anonymous referees for the many improvements to the article that they furnished.}

\author[Dina]{Bogdan Dina}
\address{
  Bogdan Dina,
  Eistein Institute of Mathematics, The Hebrew University of Jerusalem Givat Ram. Jerusalem, 9190401, Israel\\ (Former University: Institut für Algebra und Zahlentheorie,
  Universität Ulm,
  Helmholtzstrasse 18,
  D-89081 Ulm,
  Germany)
}
\email{bogdan.dina@mail.huji.ac.il}

\author[Ionica]{Sorina Ionica}
\address{
  Sorina Ionica,
  Laboratoire MIS,
  Universite de Picardie,
  33 Rue St-Leu,
  80039 Amiens Cedex 1,
  France
}
\email{sorina.ionica@u-picardie.fr}

\author[Sijsling]{Jeroen Sijsling}
\address{
  Jeroen Sijsling,
  Institut für Algebra und Zahlentheorie,
  Universität Ulm,
  Helmholtzstrasse 18,
  D-89081 Ulm,
  Germany
}
\email{jeroen.sijsling@uni-ulm.de}


\maketitle
\addtocontents{toc}{\protect\setcounter{tocdepth}{1}}

\section*{Introduction}\label{sec:introduction}
\numberwithin{equation}{section}

Because of their singular arithmetic properties and their cryptographic applications, curves of low genus whose Jacobian is simple and admits complex multiplication (CM) have historically been at the forefront of research on algebraic curves. By Shimura and Taniyama's theory of complex multiplication, it is well known that the invariants of these curves generate certain abelian extensions of CM fields. There is a wide literature on computing these invariants as well as models for genus 1 and genus 2 hyperelliptic curves with CM (see for instance \cite{BBEL, Enge-genus1, engethome, streng-algorithm}). In recent years, the frontier has shifted to genus 3, where both hyperelliptic and non-hyperelliptic curves can be considered. Hyperelliptic CM curves over $\Q$ whose endomorphism algebras are not generated by their automorphisms were constructed in \cite{weng2}, and this method was generalized in \cite{BILV}. Starting from the observation that the Galois conjugates of a hyperelliptic Jacobian are once more hyperelliptic Jacobians, \cite{DI} computes the Rosenhain and Shioda invariants of hyperelliptic curves in the Galois orbit in order to exhibit equations over the corresponding class fields.

Non-hyperelliptic CM curves were first studied in the form of Picard curves, whose first construction goes back to \cite{koikeweng}, the methods of which were recently generalized in \cite{somoza}. The reduction properties of CM Picard curves have also been studied in detail \cite{kilicer1}. The consideration of general non-hyperelliptic CM curves was taken up in the context of work of Kılıçer, who in her thesis \cite{kilicer3} determined the full list of sextic CM fields for which there exists a CM curve with field of moduli $\Q$. Explicit defining equations of these curves were found in \cite{resnt} in the hyperelliptic case and \cite{KLRRSS} in the non-hyperelliptic case, which included 19 non-hyperelliptic, non-Picard CM curves. Reduction properties in the general case were studied in the works \cite{wincm,kilicer2}.

The explicit defining equations in the works \cite{resnt,KLRRSS,weng2} mentioned above are \defi{heuristic}, in the sense that while the results obtained are exceedingly likely to be correct, their rigorous verification, available in theory by using the methods in \cite{cmsv-endos}, is still open because of the long running time of the algorithms in \loccit\ The current work, which will obtain defining equations and invariants of CM curves over $\Qbar$, will also take this heuristic approach throughout. Still, we will give multiple reasons to believe in the correctness of its results.

In \cite{kilicer3}, Kılıçer determined the sextic CM fields $K$ for which there exists a curve $X$ over $\Q$ with CM by $\Z_K$. For all of these fields $K$, there turns out to be a unique such curve $X$ over $\Q$ up to $\Qbar$-isomorphism. Often this is simply the case because there exists only a single $\Qbar$-isomorphism class of principally polarized abelian threefolds with CM by $\Z_K$ at all, see \cite[Table 1]{KLRRSS}. Yet there are also cases where, while there is only a single curve $X$ defined over $\Q$ with CM by $\Z_K$ up to $\Qbar$-isomorphism, there are other $\Qbar$-isomorphism classes of curves with such CM (defined over proper extensions of $\Q$). In these cases, it still turns out that \emph{all} CM curves over $\Qbar$ thus obtained, and not only the curve $X$ defined over $\Q$, are hyperelliptic. This happens because we have $K \supset \Q (i)$ for all $K$ for which there exist multiple $\Qbar$-isomorphism classes of CM curves. Indeed, the classification of possible automorphism groups of hyperelliptic and non-hyperelliptic curves of genus 3 (see for example \cite{LLRS}) shows that if the sextic CM field $K$ contains $\Q (i)$, then any curve whose Jacobian has primitive CM by $K$ automatically hyperelliptic.

The goal of this paper is to explore the opposite of this phenomenon: We ask ourselves whether there are sextic CM fields $K$ for which there exist both a hyperelliptic \emph{and} a non-hyperelliptic curve whose endomorphism ring is isomorphic to the maximal order $\Z_K$ of $K$.

To achieve our goal, we needed to develop and implement new computational methods; indeed, in order to check the existence of both types of curves for all 547,156 sextic CM fields included in the LMFDB \cite{lmfdb}, we need to be able to inspect the curves over a given sextic CM field $K$ in a rapid fashion. To this end, we first note that the property of a CM curve being hyperelliptic does not change under Galois conjugation (cf.\ \cite[Theorem 4.2]{DI}), and as is mentioned in Remark \ref{rem:factor}(ii), considering curves up to Galois conjugation often cuts down their number by a two- or even three-digit factor. The key to making possible this consideration of CM curves up to Galois conjugation without actually determining these curves algebraically (which is too laborious by far) is to use the Main Theorem of complex multiplication, given in the form of Theorem \ref{thm:srecip} in this article. This theorem shows that given a fixed primitive CM type $\Phi$, the Galois conjugates of a curve with CM of type $\Phi$ over the reflex field of $\Phi$ correspond to the image, call it $H$, of the reflex type norm of $\Phi$ in the Shimura class group $\Cc_K$. This latter group acts transitively on the set CM curves with CM type $\Phi$ up to isomorphism; in fact, this set is a torsor under $\Cc_K$ by Proposition \ref{prop:torsor}. Using Proposition \ref{prop:bb_square_image_rtypenorm} and \ref{prop:bb_dsquare_image_rtypenorm}, we show that the quotient $\Cc_K / H$ is a quotient of a group $\Cc_K / \widetilde{H}$ of small exponent, where $\widetilde{H} \subset H$ is a group that we can determine starting from $\Cc_K$ without actually computing the group $H$ itself, which is again too involved. An exact and explicit statement is given in Theorem \ref{thm:boundexp}. On the way to this theorem, we prove several results on the CM types of sextic CM fields and their reflex types in Section \ref{sec:cm}.

This means that the set of isomorphism classes up to Galois conjugation of CM curves that admit CM of type $\Phi$ is exhausted by representatives of the relatively small group $\Cc_K / \widetilde{H}$. Moreover, our techniques allow us to describe these representatives in terms of pairs $(\aa, \xi)$ considered in \cite{shimura-taniyama}, as reviewed in Section \ref{sec:geometric_relevance}. Algorithm \ref{alg:algo3} shows how to determine these pairs explicitly, as well as how to calculate corresponding small period matrices. Doing this efficiently uses techniques that were also considered in \cite{engethome} in lower genus, which we briefly elaborate upon. Checking whether an even theta-null value is zero then allows us to see which of these small period matrices give rise to hyperelliptic or non-hyperelliptic curves.

Our first main result is the following.

\begin{MT1}\label{mainresult:1}
  Heuristically, there are 14 sextic CM fields $K$ in the LMFDB for which there exist both a hyperelliptic and a non-hyperelliptic curve whose Jacobian has primitive complex multiplication by the maximal order $\Z_K$ of $K$. For all of these fields $K$ we have that $\Gal (K \ext \Q) \simeq C_2^3 \rtimes S_3$.
\end{MT1}

Though we cannot seriously formulate a conjecture in this direction for lack of mathematical rigor, circumstantial evidence does to some extent suggest that the sextic CM fields obtained in Main Result 1 are in fact all of their kind. Indeed, the largest absolute value of the discriminant of the fields in Main Result 1 equals $5.40 \cdot 10^{10}$, whereas the largest such value for the 494,386 sextic CM fields in the LMFDB with $\Gal (K \ext \Q) \simeq C_2^3 \rtimes S_3$ equals $1.78 \cdot 10^{17}$. Sorted by discriminant, the index of the field with largest discriminant in Main Result 1 equals $35,447$.

We also state an additional result on hyperelliptic curves. For more detailed information, see Section \ref{sec:fields}, and in particular Table \ref{tab:results}.

\begin{MT2}\label{mainresult:2}
  Heuristically, including the fields mentioned in Main Result 1, there are 3,422 CM fields $K$ in the LMFDB for which there exists a hyperelliptic curve whose Jacobian has primitive complex multiplication by the maximal order $\Z_K$ of $K$. Of these fields, 348 (resp.\ 3,057, resp.\ 17) have Galois group isomorphic to $C_6$ (resp.\ $D_6$, resp. $C_2^3 \rtimes S_3$). We have $\Q (i) \subset K$ for all but 19 of these fields $K$. Among the exceptional cases, 2 (resp.\ 17) have Galois group isomorphic to $C_6$ (resp. $C_2^3 \rtimes S_3$).
\end{MT2}

Families of fields containing $\Q(i)$ are quickly found, for example by considering those defined by polynomials of the form $x^6 + d^2$. Examples of hyperelliptic curves with CM by such fields were already computed in \cite{weng2}. In this sense, the exceptional cases with $\Q (i) \not\subset K$ are also the more interesting ones. For the 2 cyclic cases among them, equations for corresponding hyperelliptic curves were known \cite{shimura-taniyama,resnt}. By contrast, the 17 exceptional fields with Galois group $C_2^3 \rtimes S_3$ are completely new. Note that these cases include the fields from Main Result 1.

Besides determining the fields involved, we can also find corresponding invariants by using the fast algorithms from \cite{labrande}. Our final main result even gives a defining equation for the field in the Main Result 1 with smallest absolute discriminant.

\begin{MT3}
  Let $K$ be the CM field of discriminant $-1 \cdot 2^8 \cdot 359^2$ defined by the polynomial $t^6 + 10 t^4 + 21 t^2 + 4$, and let $r$ be a zero of the polynomial $t^4 - 5 t^2 - 2 t + 1$. Consider the hyperelliptic curve
  \begin{equation}\label{eq:Xsmall}
    \begin{aligned}
      X : \quad \scriptstyle{y^2 \; = \;} & \scriptstyle{x^8 + (-28 r^3 - 4 r^2 + 132 r + 84) x^7 + (-600 r^3 - 160 r^2 + 2920 r + 2044) x^6} \\
      & \scriptstyle{+ (-3532 r^3 - 940 r^2 + 17224 r + 11944) x^5 + (9040 r^3 + 2890 r^2 - 44860 r - 31460) x^4} \\
      & \scriptstyle{+ (167536 r^3 + 49480 r^2 - 824532 r - 576212) x^3 + (-226976 r^3 - 64932 r^2 + 1113648 r + 776872) x^2} \\
    & \scriptstyle{+ (-244204 r^3 - 69572 r^2 + 1197716 r + 835300) x + (319956 r^3 + 94725 r^2 - 1575062 r - 1100801)}
    \end{aligned}
  \end{equation}
  and the smooth plane quartic curve
  \begin{equation}\label{eq:Ysmall}
    \begin{aligned}
      Y : \quad & \scriptstyle{(14106 r^3 - 150652 r^2 + 185086 r + 292255) x^4 + (-171112 r^3 + 44200 r^2 + 916008 r + 93360) x^3 y} \\
      & \scriptstyle{+ (-120788 r^3 + 49032 r^2 + 382244 r + 300708) x^3 z + (467744 r^3 - 209864 r^2 - 2160704 r + 183416) x^2 y^2} \\
      & \scriptstyle{+ (-72248 r^3 + 64768 r^2 + 347488 r - 362984) x^2 y z + (5720 r^3 - 12378 r^2 - 15628 r + 50692) x^2 z^2} \\
      & \scriptstyle{+ (-512608 r^3 + 349824 r^2 + 2423616 r - 580448) x y^3 + (202192 r^3 - 151024 r^2 - 1180320 r + 403568) x y^2 z} \\
      & \scriptstyle{+ (6512 r^3 - 11272 r^2 + 178120 r - 71336) x y z^2 + (-11832 r^3 + 12268 r^2 - 844 r + 1376) x z^3} \\
      & \scriptstyle{+ (263424 r^3 - 176880 r^2 - 1159232 r + 335040) y^4 + (-201216 r^3 + 100448 r^2 + 856096 r - 249632) y^3 z} \\
      & \scriptstyle{+ (62112 r^3 + 1984 r^2 - 226512 r + 71624) y^2 z^2 + (- 12520 r^3 - 13112 r^2 + 27736 r - 5360) y z^3} \\
      & \scriptstyle{+ (1526 r^3 + 2411 r^2 - 658 r + 197) z^4 \; = \; 0.}
    \end{aligned}
  \end{equation}
  Heuristically, there exists an isogeny of degree $2$ between the Jacobians of $X$ and $Y$, and both have CM by the maximal order $\Z_K$.
\end{MT3}

The paper is structured as follows. In Sections \ref{sec:cm} and \ref{sec:shimura} we review the theory on CM fields and their Shimura class groups that we need. In particular, we prove general results on the transitivity of the Galois action on CM types and on the image of the reflex type norm that allow us to determine a small set of representatives of principally polarized abelian threefolds with CM by a given ring of integers $\Z_K$ up to Galois conjugation over the reflex field. In Section \ref{sec:lmfdb} we use these results and further speedups to check the 547,156 sextic CM fields in the LMFDB for the existence of a corresponding hyperelliptic curve, which leads to Main Results 1 and 2.

Section \ref{sec:equations} discusses techniques for determine explicit defining equations, and includes the proof of the third Main Result. We conclude the paper by some discussions around the relevance of the André--Oort conjecture to our considerations in Section \ref{sec:andreoort}.

A full implementation of our techniques in \Magma\ \cite{magma} is an essential part of our results. It is available online at \cite{dis-github}.

\subsection*{Notations and conventions}

In this article a \defi{curve} over a field $k$ is a separated and geometrically integral scheme of dimension $1$ over $k$. Given an affine equation for a curve, we will identify it with the smooth projective curve that has the same function field. The Jacobian of a curve $X$ is denoted by $\Jac (X)$. We denote the cyclic group with $n$ elements by $C_n$ and the dihedral group with $2 n$ elements by $D_n$.

When the context allows it, we often use the abbreviated notation $A$ for a principally polarized abelian variety $(A, E)$, as well as using the abbreviation ``ppav'' to stand for ``principally polarized abelian variety''. Finally, as in the introduction, we will call a curve $X$ of genus $g$ over an algebraically closed field a \defi{CM curve} if the endomorphism ring of its Jacobian $\Jac (X)$ is isomorphic to an order in a CM \defi{field} of degree $2 g$, so that by this definition, the CM type of a CM curve is primitive.

\numberwithin{equation}{subsection}
\addtocontents{toc}{\protect\setcounter{tocdepth}{3}}
\tableofcontents

\section{Theoretical background}\label{sec:cm}

For more on the general theory of complex multiplication, we refer to \cite[\S 4]{streng-algorithm}. We supplement this background with some additional considerations, most of them specific to genus $3$, that we will need for later results.

\subsection{Structure of sextic CM fields}

A number field $K$ is called a \defi{CM field} if $K$ is a totally imaginary quadratic extension of a totally real number field. Given $K$, the latter field is determined uniquely; we denote it by $K_0$. As \cite[p6]{Lang} shows, for any CM field $K$ there exists a unique element $\rho \in \Aut (K)$ such that
\begin{equation}
  \iota (\rho (x)) = \overline{\iota(x)}
\end{equation}
for all embeddings $\iota:K \into \C$. We call $\rho$ the \defi{complex conjugation} on $K$. Moreover, the Galois closure of a CM field is once again a CM field. We will consider sextic CM fields in this article. The corresponding Galois groups can be described as follows:

\begin{theorem}\label{thm:sexticCM}
  Let $K$ be sextic CM field, with Galois closure $L$. Then $G = \Gal(L \ext \Q)$ is isomorphic to one of the following groups:
  \begin{enumerate}
    \item $C_6$;
    \item $D_6$;
    \item $C_2^3 \rtimes C_3$;
    \item $C_2^3 \rtimes S_3$.
  \end{enumerate}
  In the latter two cases, the action of $C_3$ and $S_3$ on $C_2^3$ is given by permutation of the indices. Each possible group $G$ above admits a unique embedding $\iota : G \to S_6$ up to conjugation in $S_6$, under which they become the groups 6T1, 6T3, 6T6, 6T11 from \cite{gap-database}.
\end{theorem}

\begin{proof}
  The first part follows from \cite[Sec. 5.1.1]{Dodson}; see also \cite[Proposition 2.1]{wincm}. The second is a one-off calculation with the conjugacy classes of subgroups of $S_6$, for example by using GAP \cite{gap-database}.
\end{proof}

The notation 6TX for the groups in Theorem \ref{thm:sexticCM} can be used when searching for corresponding fields in the LMFDB \cite{lmfdb}.

\begin{remark}\label{rem:H}
  The second part of Theorem \ref{thm:sexticCM} in combination with Galois theory shows that we may assume that under the chosen embedding $\iota : G \to S_6$ the subgroup $H = \Gal (L \ext K)$ is the stabilizer of $1$.
\end{remark}

\begin{example}
  The sextic CM field with smallest absolute discriminant in the LMFDB whose Galois group is isomorphic to $C_6$ is \href{https://www.lmfdb.org/NumberField/6.0.16807.1}{$\Q (\zeta_7)$}. The field of smallest absolute discriminant with Galois group $D_6$ is defined by \href{https://www.lmfdb.org/NumberField/6.0.309123.1}{$x^6 - 3 x^5 + 10 x^4 - 15 x^3 + 19 x^2 - 12 x + 3$}, that for $C_2^3 \rtimes C_3$ by \href{https://www.lmfdb.org/NumberField/6.0.400967.1}{$x^6 - 2 x^5 + 5 x^4 - 7 x^3 + 10 x^2 - 8 x + 8$}, and that for $C_2^3 \rtimes S_3$ by \href{https://www.lmfdb.org/NumberField/6.0.503792.1}{$x^6 - 3 x^5 + 9 x^4 - 13 x^3 + 14 x^2 - 8 x + 2$}.
\end{example}

\subsection{CM types}

As above, let $K$ be a CM field with complex conjugation $\rho$ and Galois closure $L$. We further fix an embedding $\iota_L : L \to \C$.

\begin{definition}\label{def:CMType}
  A \defi{CM type of $K$ (with values in $L$)} is a subset $\Phi \subset \Hom (K, L)$ such that
  \begin{equation}
    \Hom (K, L) = \Phi \amalg \Phi \rho .
  \end{equation}
\end{definition}

As in the classical case, we call a CM type of $K$ \defi{primitive} if it is not induced by a CM type of a strict CM subfield. Similarly, we call two CM types $\Phi, \Phi'$ \defi{equivalent} if there exists an automorphism $\alpha \in \Aut (K)$ such that $\Phi' = \Phi \alpha$. As for example in \cite{Lang}, we also call the pair $(K, \Phi)$ a CM type.

\begin{remark}
Our choice of an embedding $\iota_L : L \to \C$ yields a map
  \begin{equation}
    \Phi \mapsto \left\{ \iota_L \circ \tau : \tau \in \Phi \right\}
  \end{equation}
  which yields a bijection between the CM types in Definition \ref{def:CMType} and the CM types of $K$ in the classical sense of a set of embeddings into $\C$. For our purposes, it is more useful to consider the former type.
\end{remark}

\begin{remark}
  Let $H = \Gal (L \ext K)$. Then the natural map
  \begin{equation}
    \begin{aligned}
      \Gal (L \ext \Q) & \to \Hom (K, L) \\
      \sigma & \mapsto \sigma|_K
    \end{aligned}
  \end{equation}
  induces a bijection $G / H \to \Hom (K, L)$. We can and therefore will consider the individual embeddings $\tau : K \to L$ in a CM type as cosets $\sigma H$ in $G / H$. Under this interpretation, a CM type is nothing but a section of the natural projection map
  \begin{equation}
    G / H \to \langle \rho \rangle \backslash G / H .
  \end{equation}
  Alternatively, it is a subset $\Phi \subset G / H$ on which the restriction of this projection map induces an bijection.

  The normalizer $N = N_G (H)$ of $H$ in $G$ acts on the set of sections $s : \langle \rho \rangle \backslash G / H \to G / H$ via right composition, and using the natural isomorphism $N / H \cong \Aut (K)$ induced by restriction, we see that the corresponding quotient is in bijection with the set of CM types up to equivalence. We determine this normalizer in the following proposition.
\end{remark}

\begin{proposition}\label{prop:norm}
  Let $G = \Gal (L \ext \Q)$ be one of the Galois groups in Theorem \ref{thm:sexticCM} and let $H = \Gal (L \ext K)$. Let $N = N_G (H)$ be the normalizer of $H$ in $G$. Then we have $N = \langle H, \rho \rangle$ except if $G = C_6$, in which case $N = G$.
\end{proposition}

\begin{proof}
  We have realized our Galois groups as the explicit subgroups 6T1, 6T3, 6T6, 6T11, and Remark \ref{rem:H} shows that we may take $H$ to be the stabilizer of $1$. We can choose our embeddings $G \to S_6$ in such a way that we have the following, as used throughout the paper:
  \begin{enumerate}
    \item $G = C_6 = \langle \sigma \rangle$: $H = 1$, $\rho = \sigma^3$.
    \item $G = D_6 = \langle \sigma, \tau \rangle$: $H = \langle \tau \rangle$, $\rho = \sigma^3$.
    \item $G = C_2^3 \rtimes C_3$: $H = \langle ((1,0,0), e), ((0,1,0), e) \rangle$, $\rho = ((1,1,1), e)$, where $e \in C_3$ is the identity.
    \item $G = C_2^3 \rtimes S_3$: $H = \langle ((1,0,0), e), ((0,1,0), e), ((0,0,0), (1,2)) \rangle$, $\rho = ((1,1,1), e)$, where $e \in S_3$ is the identity.
  \end{enumerate}
  The result is now an unenlightening and unproblematic calculation.
\end{proof}

\begin{definition}\label{def:galeq}
  There is a natural left action of the Galois group $G = \Gal (L \ext \Q)$ on CM types $\Phi$ of $K$. As subsets $\Phi \subset G / H$, we have
  \begin{equation}
    \sigma \Phi = \left\{ \sigma \phi : \phi \in \Phi \right\}
  \end{equation}
  for $\sigma \in G$. On CM types considered as sections $s$ of the projection map $G / H \to \langle \rho \rangle \backslash G / H$, the action is defined by
  \begin{equation}
    (\sigma s) (\langle \rho \rangle c H) = \sigma \cdot s (\langle \rho \rangle \sigma^{-1} c H)
  \end{equation}
  for $\sigma, c \in G$. Note that these actions are well-defined because $\rho$ is central in $G$. We call the resulting equivalence on the set of CM types of $K$ the \defi{Galois equivalence}.
\end{definition}

\begin{remark}\label{rem:galeq}
  As is shown in Proposition \ref{prop:tang}, the equivalence in Definition \ref{def:galeq} reflects Galois conjugation of abelian varieties, in the sense that if $A$ is an abelian variety that admits $\Phi$ as a CM type and $\sigma$ is an automorphism of $\C$, then the conjugate $\sigma A$ of $A$ admits $\sigma \Phi$ as a CM type.
\end{remark}

\begin{proposition}\label{prop:c6}
  Let $K$ be a sextic CM field with Galois group $C_6$. Then $K$ admits $2$ CM types up to equivalence, $1$ of them primitive and $1$ imprimitive. The same is true when replacing equivalence with Galois equivalence.
\end{proposition}

\begin{proof}
  We can identify a CM type on $K$ with a subset $S \subset C_6 = \Z / 6 \Z$ of cardinality 3 such that $S$ and $3 + S$ cover $\Z / 6 \Z$. By Proposition \ref{prop:norm}, two such CM types $S, S'$ are equivalent if they are related by a translation, so that $S' = i + S$ for some $i \in \Z / 6 \Z$, and the same is true for Galois equivalence. As is readily verified, representatives up to equivalence are given by $\left\{ 0, 1, 5 \right\}$ and $\left\{ 0, 2, 4 \right\}$. The latter CM type is imprimitive, since it is induced from the quotient $\Z / 2 \Z$ of $\Z / 6 \Z$ that corresponds to the unique CM quadratic subfield of $K$. The former type is primitive. Since the Galois equivalence does not affect primitivity, we obtain the result. See \cite[\S 3.1]{wincm} for a different point of view.
\end{proof}

\begin{proposition}\label{prop:d6}
  Let $K$ be a sextic CM field with Galois group $D_6$. Then $K$ admits $4$ CM types up to equivalence, $3$ of them primitive and $1$ imprimitive. Up to Galois equivalence, $K$ admits $2$ CM types, $1$ of them primitive and $1$ imprimitive.
\end{proposition}

\begin{proof}
  In this case we can choose a standard representation $D_6 = \langle \sigma, \tau \rangle$ and embed it into $S_6$ by identifying $\sigma$ with $(1\,2\,3\,4\,5\,6)$ and $\tau$ with $(2\,6)(3\,5)$. As we have seen in Proposition \ref{prop:norm}, the complex conjugation $\rho$ is given by the central element $\sigma^3$ and $\Gal (L \ext K) = \langle \tau \rangle$. The embeddings of $K$ into $L$ can therefore be identified with powers $\sigma^i$, or for that matter with elements $i$ of $\Z / 6 \Z$. We are in a similar situation as Proposition \ref{prop:c6}, except that the notion of equivalence is stricter, as Proposition \ref{prop:norm} shows that this time the only other CM type equivalent to a given type $\left\{ a, b, c \right\}$ is $\left\{ a + 3, b + 3, c + 3 \right\}$, which corresponds to applying complex conjugation.

  Up to equivalence, we obtain the $4$ CM types $\left\{ 0, 1, 5 \right\}$, $\left\{ 0, 1, 2 \right\}$, $\left\{ 0, 4, 5 \right\}$, $\left\{ 0, 2, 4 \right\}$. Of these types, $\left\{ 0, 2, 4 \right\}$ is induced by the unique quadratic CM subfield of $K$ and is therefore imprimitive, while the other types are primitive.

  Applying Galois equivalence allows us to multiply with $\sigma$, so as in Proposition \ref{prop:c6} we can apply arbitrary shifts to our subsets of $\Z / 6 \Z$ to our CM types. Once more this reduces us to the two types $\left\{ 0, 1, 2 \right\}$ and $\left\{ 0, 2, 4 \right\}$, the former primitive and the latter imprimitive. See \cite[\S 3.2]{wincm} for a different point of view.
\end{proof}

\begin{proposition}\label{prop:c23_c3_s3}
  Let $K$ be a sextic CM field with Galois group $C_2^3 \rtimes C_3$ or $C_2^3 \rtimes S_3$. Then $K$ admits $4$ CM types up to equivalence, which are all primitive. Up to Galois equivalence, $K$ admits $1$ CM type.
\end{proposition}

\begin{proof}
  The first statement follows as in Proposition \ref{prop:d6}, since in light of Proposition \ref{prop:norm} applying equivalence once again comes down to identifying complex conjugate CM types, leaving $4$ equivalence classes of the original $8$ types. All of these types are primitive because the group $H$ corresponding to $K$ in the notation of Proposition \ref{prop:norm} is not contained in any subgroup of $G$ of index $2$, or in other words because $K$ has no proper quadratic subfields, let alone proper CM subfields.

  As for working up to Galois equivalence, in the case $G = C_2^3 \rtimes C_3$ using the element $\sigma = ((0, 0, 0), (1\,2\,3))$ shows that
  \begin{equation}\label{eq:cosets}
    \begin{split}
                    H & = \left\{ ((*, *, 0),         e) \right\}, \\
             \sigma H & = \left\{ ((0, *, *), (1\,2\,3)) \right\}, \\
           \sigma^2 H & = \left\{ ((*, 0, *), (1\,3\,2)) \right\}, \\
               \rho H & = \left\{ ((*, *, 1),         e) \right\}, \\
        \sigma \rho H & = \left\{ ((1, *, *), (1\,2\,3)) \right\}, \\
      \sigma^2 \rho H & = \left\{ ((*, 1, *), (1\,3\,2)) \right\},
    \end{split}
  \end{equation}
  where $*$ denotes an element of $C_2$ that can be chosen freely. We see that $\Phi_0 = \left\{ H, \sigma H, \sigma^2 H \right\}$ is a CM type. Moreover, the definition of the Galois action along with that of the group structure on $G$ implies that for $n_1 = ((1, 0, 0), e)$ we have
  \begin{equation}\label{eq:neqs}
    \begin{split}
               n_1 H & = \left\{ ((*, *, 0), e) \right\} = H                     \\
        n_1 \sigma H & = \left\{ ((1, *, *), (1\,2\,3)) \right\} = \sigma \rho H \\
      n_1 \sigma^2 H & = \left\{ ((*, 0, *), (1\,3\,2)) \right\} = \sigma^2 H.
    \end{split}
  \end{equation}
  Similarly, $n_2 = ((0, 1, 0), e)$ sends $\Phi_0$ to $\left\{ H, \sigma H, \sigma^2 \rho H \right\}$ and $n_0 = ((0, 0, 1), e)$ sends $\Phi_0$ to $\left\{ \rho H, \sigma H, \sigma^2 H \right\}$. Combining the action of these three elements is enough to obtain transitivity of the Galois action on the full set of CM types $\left\{ \rho^*, \sigma \rho^*, \sigma^2 \rho^* \right\}$. The considerations for $G = C_2^3 \rtimes S_3$ are completely similar, and we can again use the element with $\sigma = ((0, 0, 0), (1\,2\,3))$ to show that
  \begin{equation}
    \begin{split}
                    H & = \left\{ ((*, *, 0),         \text{$e$ or $(1\,2)$}) \right\}, \\
             \sigma H & = \left\{ ((0, *, *), \text{$(1\,2\,3)$ or $(1\,3)$}) \right\}, \\
           \sigma^2 H & = \left\{ ((*, 0, *), \text{$(1\,3\,2)$ or $(2\,3)$}) \right\}, \\
               \rho H & = \left\{ ((*, *, 1),         \text{$e$ or $(1\,2)$}) \right\}, \\
        \sigma \rho H & = \left\{ ((1, *, *), \text{$(1\,2\,3)$ or $(1\,3)$}) \right\}, \\
      \sigma^2 \rho H & = \left\{ ((*, 1, *), \text{$(1\,3\,2)$ or $(2\,3)$}) \right\}.
    \end{split}
  \end{equation}
\end{proof}

\begin{corollary}\label{cor:trans}
  Let $K$ be a sextic CM field. Then all primitive CM types of $K$ are Galois equivalent.
\end{corollary}

\begin{proof}
  We proved this result in Propositions \ref{prop:c6}, \ref{prop:d6}, and \ref{prop:c23_c3_s3}, which cover all individual cases in Theorem \ref{thm:sexticCM}.
\end{proof}

\begin{remark}
  In genus $4$, it is no longer true that all primitive CM types are Galois equivalent. Let $K$ be an octic CM field with Galois group $C_8$, for example $\Q (\zeta_{32} + \zeta_{32}^{15})$. Then (with notation as in the case $C_6$ above) the CM types $\left\{ 0, 1, 2, 3 \right\}$ and $\left\{ 0, 1, 2, 6 \right\}$ are primitive, yet they are not related even when combining the two equivalences.
\end{remark}

In the next section we will consider representatives of the equivalence classes of CM types and their reflex CM types more explicitly.

\subsection{Reflex CM types}\label{sec:reflex_cm_types}

Given a CM type $(K, \Phi)$, there is a \defi{reflex CM type} $(K^r, \Phi^r)$ to $(K, \Phi)$, which is constructed as follows. We can lift $\Phi$ to a CM type $\Phi_L$ on $L$, where elements in $\Phi_L$ are identified by elements in $\Gal(L  | \Q)$. Inverting elements in $\Phi_L$ gives rise to a CM type on $L$ denoted by
\begin{align}
  \Phi_L^{-1} = \{\phi^{-1}: \phi \in \Phi_L\}.
\end{align}
As in \cite[Lemma 2.2]{Lang}, we define the CM field $K^r$ as the fixed field of the group
\begin{equation}
  H^r = \{\sigma\in \Gal(L \ext \Q): \sigma\Phi_L = \Phi_L\},
\end{equation}
and $\Phi^r$ to be the unique (primitive) CM type on $K^r$ that induces $\Phi_L^{-1}$.

To obtain explicit descriptions of reflex CM types when $K$ is sextic, we need one more definition.

\begin{definition}\label{def:dist}
  Let $K$ be a sextic CM field with Galois group $C_6$ or $D_6$. Then we define the \defi{distinguished CM type} $\Psi$ of $K$ to be the CM type $\left\{ \id|_K, \sigma|_K, \sigma^{-1}|_K \right\}$ that corresponds to the set $\left\{ 0, 1, 5 \right\}$ in the notation of Propositions \ref{prop:c6} and \ref{prop:d6}. Note that $\Psi$ is intrinsic to $K$: In other words, it does not depend on the choice of element $\sigma$ of order $6$.
\end{definition}

\begin{proposition}\label{prop:reflex_c6}
  Let $K$ be a sextic CM field with Galois group $C_6$, with generator $\sigma$ as in the proof of Proposition \ref{prop:norm}, and let $\Psi$ be the distinguished CM type of $K$. Then
  \begin{equation}
    \Phi_1 = \Psi = \{\id|_K, \sigma|_K, \sigma^{-1}|_K\}
  \end{equation}
  is the single primitive CM type of $K$ up to equivalence, and
  \begin{equation}
    \Phi_2 = \{\id|_K, \sigma^2|_K, \sigma^{-2}|_K\}
  \end{equation}
  is the single imprimitive CM type of $K$ up to equivalence.

  We have $(K^r, \Phi^r) = (K, \Phi)$ for all primitive CM types $\Phi$, and the reflex CM type $(K^r, \Phi^r)$ of an imprimitive CM type $\Phi$ of $K$ is the restriction of $(K, \Phi)$ to the quadratic CM subfield of $K$.
\end{proposition}

\begin{proof}
  The first part is a consequence of Proposition \ref{prop:norm}: The second follows from the fact that the left and right stabilizers of $\Phi$ coincide, and because the reflex of an imprimitive CM type coincides with that of the primitive CM type that induces it.
\end{proof}

To deal with the case $G = D_6$, we first prove the following general statement.

\begin{proposition}\label{prop:conjtype}
  Let $\Phi, \Psi$ be two CM types of a given field $K$, and suppose that $\Psi = \sigma \Phi$ for $\sigma \in G = \Gal (L \ext \Q)$. Let $(K^r, \Phi^r)$ be the reflex CM type of $(K, \Phi)$. Then the reflex CM type of $(K, \Psi)$ is given by
  \begin{equation}
    (\sigma (K^r), \Phi^r \sigma^{-1}) ,
  \end{equation}
  where
  \begin{equation}
    \Phi^r \sigma^{-1} = \left\{ (\phi \sigma^{-1})|_{\sigma (K^r)} : \phi \in \Phi^r \right\} .
  \end{equation}
\end{proposition}

\begin{proof}
  For the extensions of $\Phi$ and $\Psi$ to $L$ we have $\Psi_L = \sigma \Phi_L$. For the corresponding left stabilizers $H_{\Phi}$ and $H_{\Psi}$ we therefore have $H_{\Psi} = \sigma H_{\Phi} \sigma^{-1}$, which already shows that the reflex field of $\Psi$ equals $\sigma (K^r)$. By construction, we have
  \begin{equation}
    \Phi_L^{-1} = \coprod_{\phi \in \Phi^r} \phi H_{\Phi},
  \end{equation}
  so that
  \begin{equation}
    \Psi_L^{-1} = \Phi_L^{-1} \sigma^{-1} = \coprod_{\phi \in \Phi^r} \phi H_{\Phi} \sigma^{-1} = \coprod_{\phi \in \Phi^r} \phi \sigma^{-1} H_{\Psi} .
  \end{equation}
  Restricting to the reflex field of $\Psi$, we obtain the statement of the proposition.
\end{proof}

\begin{proposition}\label{prop:reflex_d6}
  Let $K$ be a sextic CM field with Galois group $D_6$, with dihedral generator $\sigma$ as in the proof of Proposition \ref{prop:norm}, and let $\Psi$ be the distinguished CM type of $K$. Then
  \begin{equation}
    \begin{aligned}
      \Phi_1 & = \Psi &&= \{\id|_K, \sigma|_K, \sigma^{-1}|_K\}, \\
      \Phi_2 & = \sigma \Psi &&=  \{\id|_K, \sigma|_K, \sigma^2|_K\}, \\
      \Phi_3 & = \sigma^{-1} \Psi &&= \{\id|_K, \sigma^{-1}|_K, \sigma^{-2}|_K\}
    \end{aligned}
  \end{equation}
  are representatives of the primitive CM types of $K$ up to equivalence, and
  \begin{equation}
    \Phi_4 = \{\id|_K, \sigma^2|_K, \sigma^{-2}|_K\}
  \end{equation}
  is the single imprimitive CM type of $K$ up to equivalence.

  We have
  \begin{equation}
    \begin{aligned}
      (K_1^r, \Phi_1^r) &= (K, \Psi) &&= \left( K, \left\{ \id|_K, \sigma|_K, \sigma^{-1}|_K \right\} \right), \\
      (K_2^r, \Phi_2^r) &= (\sigma (K), \Psi \sigma^{-1}) &&= \left( \sigma (K), \left\{ \id|_K, \sigma^{-1}|_K, \sigma^{-2}|_K \right\} \right), \\
      (K_3^r, \Phi_3^r) &= (\sigma^{-1} (K), \Psi \sigma) &&= \left(\sigma^{-1} (K), \left\{ \id|_K, \sigma|_K, \sigma^2|_K \right\} \right)
    \end{aligned}
  \end{equation}
  and the reflex $(K_4^r, \Phi_4^r)$ of the imprimitive CM type of $K$ is the restriction of $(K, \Phi_4)$ to the quadratic CM subfield of $K$.
\end{proposition}

\begin{proof}
  The first part of the proposition is a consequence of Proposition \ref{prop:norm}. Since in our notation from Proposition \ref{prop:norm} we have $\Gal (L \ext K) = \langle \tau \rangle$, the lift of the CM type $\Psi$ to $L$ is given by $\Psi_L = \left\{ e, \tau, \sigma, \sigma \tau, \sigma^{-1}, \sigma^{-1} \tau \right\}$, so that $\Psi_L^{-1} = \Psi_L$. This implies that the reflex of $(K, \Psi)$ is given by $(K, \Psi)$ itself. Proposition \ref{prop:conjtype} then shows the result for the primitive CM types $\Phi_i$. The statement for the imprimitive CM type follows as in Proposition \ref{prop:reflex_c6}.
\end{proof}

The reflex fields for the remaining Galois groups are described in the upcoming propositions.

\begin{proposition}\label{prop:reflex_c23_c3_s3}
  Let $K$ be a sextic CM field with Galois group $C_2^3 \rtimes C_3$, and let $\sigma = ((0, 0, 0), (1\,2\,3))$ and $\rho = ((1, 1, 1), e)$ as in Proposition \ref{prop:norm}. Then
  \begin{equation}
    \begin{aligned}
      \Phi_1 & = \{\id|_K, \sigma|_K, \sigma^2|_K\}, && \Phi_2 = \{\id|_K, \sigma\rho|_K, \sigma^2|_K\},\\\Phi_3 & = \{\id|_K, \sigma|_K, \sigma^2\rho|_K\}, && \Phi_4  = \{\id|_K, \sigma\rho|_K, \sigma^2\rho|_K \}
    \end{aligned}
  \end{equation}
  are representatives of the (primitive) CM types of $K$ up to equivalence.

  The reflex field $K_i^r$ of $(K, \Phi_i)$ is fixed by the group $H^r_i \subset \Gal(L | \Q)$, where
  \begin{equation}
    H_1^r = \langle \sigma\rangle,~H_2^r = \langle \sigma n_0n_1 \rangle,~H_3^r = \langle \sigma n_1n_2 \rangle,~H_4^r = \langle \sigma n_0n_2 \rangle,
  \end{equation}
  with $n_i$ as defined in the proof of Proposition \ref{prop:c23_c3_s3}. We have
  \begin{equation}
    K_2^r = n_1 (K_1^r),~K_3^r = n_2 (K_1^r),~K_4^r = n_1 n_2 (K_1^r).
  \end{equation}
  The reflex CM types $\Phi_i^r$ are given by
  \begin{equation}
    \Phi_1^r = \Phi_2^r = \Phi_3^r = \Phi_4^r = \{\id|_{K_i^r}, n_1|_{K_i^r}, n_2|_{K_i^r}, n_1n_2|_{K_i^r}\}.
  \end{equation}
\end{proposition}

\begin{proof}
  The first part of the proposition is a consequence of Proposition \ref{prop:norm}. Let $H = \left\{ e, n_1, n_2, n_1 n_2 \right\}$ be the subgroup of the Galois group that corresponds to $K$. First consider the CM type $\Phi_1$. The induced CM type $\Phi_{1, L}$ on $L$ is the union
  \begin{equation}
    H \cup \sigma H \cup \sigma^2 H = \left\{ ((*, *, 0), e) \right\} \cup \left\{ ((0, *, *), (1\,2\,3)) \right\} \cup \left\{ ((*, 0, *), (1\,3\,2)) \right\} .
  \end{equation}
  From this explicit presentation, one obtains that the left stabilizer $H_1^r$ of $\Phi_{1, L}$ is generated by $\sigma$. Using equalities similar to \eqref{eq:neqs} shows that $\Phi_2 = n_1 \Phi_1$, $\Phi_3 = n_2 \Phi_1$, and $\Phi_4 = n_1 n_2 \Phi_1$. The corresponding stabilizers are therefore generated by $n_1 \sigma n_1^{-1} = \sigma n_0 n_1$, $n_2 \sigma n_2^{-1} = \sigma n_1 n_2$, and $n_1 n_2 \sigma (n_1 n_2)^{-1} = \sigma n_0 n_2$.

  The embeddings in the reflex CM type of $\Phi_1$ are in bijective correspondence with the elements of
  \begin{equation}
    \Phi_{1, L}^{-1} = H^{-1} \cup H^{-1} \sigma^{-1} \cup H^{-1} \sigma^{-2} = H \cup H \sigma^{-1} \cup H \sigma^{-2}
  \end{equation}
  up to the action of the right stabilizer $H_1^r = \langle \sigma \rangle$. These are therefore represented by the elements of $H$, which yields the second statement of the proposition for $\Phi_1$. Since $\Phi_2 = n_1 \Phi_1$, $\Phi_3 = n_2 \Phi_1$, and $\Phi_4 = n_1 n_2 \Phi_1$, representatives of the corresponding inverse CM types up to the right action of the corresponding stabilizers are given by $n_1 H, n_2 H$, and $n_1 n_2 H$. These are all equal to $H$, and therefore we obtain the second statement for all CM types $\Phi_i$.
\end{proof}

\begin{proposition}\label{prop:reflex_c23_s3}
  Let $K$ be a sextic CM field with Galois group $C_2^3 \rtimes S_3$ and take $\sigma = ((0, 0, 0), (1\,2\,3))$, $\tau = ((0, 0, 0), (1\,2))$, and $\rho = ((1, 1, 1), e)$. Then
  \begin{equation}
    \begin{aligned}
      \Phi_1 & = \{\id|_K, \sigma|_K, \sigma^2|_K\}, && \Phi_2 = \{\id|_K, \sigma\rho|_K, \sigma^2|_K\},\\\Phi_3 & = \{\id|_K, \sigma|_K, \sigma^2\rho|_K\}, && \Phi_4  = \{\id|_K, \sigma\rho|_K, \sigma^2\rho|_K \}
    \end{aligned}
  \end{equation}
  are representatives of the (primitive) CM types of $K$ up to equivalence.

  The reflex field $K_i^r$ of $(K, \Phi_i)$ is fixed by the group $H^r_i \subset \Gal(L | \Q)$, where
  \begin{equation}
  H_1^r = \langle \sigma, \tau \rangle,~H_2^r = \langle \sigma n_0n_1, \tau n_1 n_2 \rangle,~H_3^r = \langle \sigma n_1n_2, \tau n_1 n_2 \rangle,~H_4^r = \langle \sigma n_0 n_2, \tau \rangle,
  \end{equation}
  with $n_i$ as defined in the proof of Proposition \ref{prop:c23_c3_s3}. We have
  \begin{equation}
    K_2^r = n_1 (K_1^r),~K_3^r = n_2 (K_1^r),~K_4^r = n_1 n_2 (K_1^r) .
  \end{equation}
  The reflex CM types $\Phi_i^r$ are given by
  \begin{equation}
    \Phi_1^r = \Phi_2^r = \Phi_3^r = \Phi_4^r = \{\id|_{K_i^r}, n_1|_{K_i^r}, n_2|_{K_i^r}, n_1n_2|_{K_i^r}\}.
  \end{equation}
\end{proposition}

\begin{proof}
  The proof is similar to that of the previous proposition.
\end{proof}

\subsection{Construction of abelian varieties with CM}\label{sec:geometric_relevance}

In this section we review the construction of abelian varieties with CM type and explore the geometric relevance of the Galois equivalence defined in \ref{def:galeq}. We develop the tools that we need for the upcoming section on abelian varieties, and prove the statement in Remark \ref{rem:galeq}, as well as a lower bound on the degree of the field of moduli in terms of the CM type in Corollary \ref{cor:fom}.

Let $K$ be a CM field of degree $2 g$. As in \cite{shimura-taniyama, spallek, wamelen}, we consider pairs $(\a, \xi)$, where $\a$ is a fractional $\Z_K$-ideal and where $\xi \in K$ is a totally imaginary element such that $(\xi) = (\a \overline{\a} \Df_{K | \Q})^{-1}$. Such a pair $(\a, \xi)$ uniquely determines a CM type $\Phi$ that consists of those embeddings $\phi$ of $K$ into $\C$ for which the imaginary part of $\phi (\xi)$ is positive. In what follows, we will consider the pairs $(\aa, \xi)$ and the corresponding triples $(\Phi, \aa, \xi)$ interchangeably. Either gives rise to a ppav $A (\a, \xi) = (\C^g / \Phi (\a), E)$ over $\CC$ whose endomorphism ring is isomorphic to $\Z_K$. The polarization $E$ that comes with $A (\a, \xi)$ is induced by the trace pairing on $K$.

Conversely, given a ppav $A$ over $\CC$ whose endomorphism ring is isomorphic to $\Z_K$, we can consider the representation $\rho$ of $K$ on the tangent space of $A$ at $0$ after choosing some isomorphism $\End (A) \otimes \Q \simeq K$. Then $\rho$ is diagonalizable, which yields a set of embeddings $\Phi$ of $K$ into $\C$ such that
\begin{equation}
  \rho \cong \phi_1 \oplus \cdots \oplus \phi_g.
\end{equation}
The set $\Phi$ is then a primitive CM type of $K$ because we insisted that $\Z_K$ be the full endomorphism ring of $A$. (If $\Phi$ were non-primitive, then $A$ would be a power of an elliptic curve up to isogeny, so that the endomorphism ring would even contain zero divisors.) Note that the representation of $K$ depends on the chosen isomorphism, which means that given the ppav $A$, only the \emph{equivalence class} of the CM type $\Phi$ is well-defined.

The following is shown in \cite[Theorem 4.2]{streng-algorithm}:
\begin{proposition}\label{prop:isom_with_CM}
  Let $K$ be a CM field of degree $2 g$, and let $\Phi$ be a fixed \emph{primitive} CM type of $K$. Then the association
  \begin{equation}
    (\a, \xi) \mapsto A (\a, \xi) = (\C^g / \Phi (\a), E)
  \end{equation}
  defined above yields a bijection between the set of pairs $(\a, \xi)$ with associated CM type $\Phi$ up to the equivalence
  \begin{align}\label{def:equiv_inside_cmtype_equiv}
    \text{$(\a, \xi) \sim (\a', \xi')$ if $(\a', \xi') = (\gamma \a, (\gamma \overline{\gamma})^{-1} \xi)$ for $\gamma \in K^*$}
  \end{align}
  and the set of isomorphism classes of principally polarized abelian varieties that admit CM by $\Z_K$ of type $\Phi$ up to equivalence.
\end{proposition}

\begin{definition}\label{def:auteq}
  We say that two pairs $(\aa, \xi)$ and $(\aa', \xi')$ are \emph{equivalent} if there exists an element $\alpha \in \Aut (K)$ such that $(\alpha^{-1} (\aa), \alpha (\xi))$ and $(\aa', \xi')$ are equivalent in the sense of Proposition \ref{prop:isom_with_CM}.
\end{definition}

\begin{remark}
  Note that this new notion of equivalence from Definition \ref{def:auteq} makes it possible for two pairs $(\aa, \xi)$ with different associated CM type to be equivalent.
\end{remark}

The following is shown in \cite[Proposition 4.11]{streng-algorithm}:
\begin{proposition}
  Two pairs $(\aa, \xi)$ and $(\aa', \xi')$ are equivalent if and only if $A (\aa, \xi)$ and $A (\aa', \xi')$ are isomorphic as principally polarized abelian varieties.
\end{proposition}

Recall that we have fixed an embedding of the Galois closure $L$ into $\C$.

\begin{proposition}\label{prop:tang}
  Let $A$ be a principally polarized abelian variety over $\C$ with CM by $K$ of type $\Phi$ up to equivalence, and let $\sigma \in \Aut (\C)$. Denoting the restriction of $\sigma$ to $L$ by $\sigma$ again, we have that the conjugate principally polarized abelian variety $\sigma A$ has CM by $K$ of type $\sigma \Phi$ up to equivalence.
\end{proposition}

\begin{proof}
This follows from the fact that the formation of the tangent space is functorial. Alternatively, if $T \in M_g (\C)$ is the tangent representation of a given endomorphism $\alpha$ with respect to a basis of differentials $B$ of $A$, then $\sigma T$ is a representation of an endomorphism of $\sigma A$ with respect to $\sigma B$. This means that if after our choice of embedding $K \into \End (A) \otimes \Q$ we can write the representation $\rho$ of $K$ on the tangent space of $A$ as a direct sum
  \begin{equation}
    \rho \cong \phi_1 \oplus \cdots \oplus \phi_g,
  \end{equation}
  we also obtain a representation $\sigma \rho$ of $K$ on the tangent space of $\sigma A$ given by
  \begin{equation}
    \sigma \rho \cong \sigma \phi_1 \oplus \cdots \oplus \sigma \phi_g,
  \end{equation}
  which proves the proposition.
\end{proof}

\begin{corollary}\label{cor:fom}
  Let $A$ be a principally polarized abelian variety over $\C$ with \emph{primitive} CM by $K$ up to equivalence, and suppose that $\Gal (L \ext \Q)$ is isomorphic to $D_6$ (resp.\ $C_2^3 \rtimes C_3$ or $C_2^3 \rtimes S_3$). Then the degree of the field of moduli of $A$ over $\Q$ is a multiple of $3$ (resp.\ $4$).
\end{corollary}

\begin{proof}
  This follows because the subgroup of $\Aut (\C)$ that fixes the field of moduli has to fix the primitive CM type up to equivalence of $A$ by Proposition \ref{prop:tang}, combined with the transitivity of the Galois action proved in Corollary \ref{cor:trans}.
\end{proof}

\section{The Shimura class group and the Galois action}\label{sec:shimura}

\subsection{Background}

We recall some fundamental notions. Throughout, we let $K$ be a CM field.

\begin{definition}
  The \defi{Shimura class group} $\Cc_K$ of $K$ is the abelian group of equivalence classes
  \begin{equation}\label{def:shimura_class_group}
    \Cc_K = \left \{ (\b, \beta): \b~\text{is fractional $\Z_K$-ideal},~\beta\in K^*_0~\text{totally positive with}~\bb \bbbar = \beta\Z_K \right\}/\sim
  \end{equation}
  where $(\b, \beta) \sim (\b', \beta')$ if $(\b', \beta') = (x\b, x\xbar \beta)$ for some $x\in K^*$.
\end{definition}

As was shown in \cite[\S 14.5]{Shimura}, the structure of $\Cc_K$ is given by the sequence
\begin{equation}\label{eq:exseq}
  1 \xrightarrow{~} (\Z^*_{K_0})^+/N_{K|K_0}(\Z_K^*) \xrightarrow{u~\longmapsto~(\Z_K, u)} \Cc_K \xrightarrow{(\b, \beta)~\longmapsto~\b} \Cl(K) \xrightarrow{N_{K|K_0}} \Cl^+(K_0),
\end{equation}
where $(\Z^*_{K_0})^+ \subset \Z^*_{K_0}$ is the group of totally positive units, and $\Cl^+(K_0)$ is the narrow class group of $K_0$.

\begin{remark}
  As is discussed in \cite{BGL}, the final map in \eqref{eq:exseq} is surjective if a finite prime ramifies in the extension $K \ext K_0$. This turns out to be the case for all fields considered in this article, as follows by checking that the relative different $\Dc_{K | K_0} = \left\{ \alpha - \rho(\alpha) : \alpha \in \Z_K \right\}$ is a proper ideal (also see the proof of \cite[Proposition 4.4]{streng-algorithm}). Under this hypothesis, denoting $h (K) = |\Cl (K)|$ and $h^+ (K_0) = |\Cl^+ (K_0)|$, we have
  \begin{equation}\label{eq:strengsize}
    \left|\Cc_K\right| = \frac{h(K)}{h^+(K_0)}\cdot  \left|(\Z^*_{K_0})^+/N_{K|K_0}(\Z_K^*)\right|.
  \end{equation}
\end{remark}

Let $N$ be the norm map on ideals of $K^r$. Combining this with the reflex type norm
\begin{align}\label{def:reflex_type_norm_map}
  \begin{split}
    N_{\Phi^r}: \Cl(K^r) & \rightarrow \Cl(K) \\
    [\a] & \mapsto \left[ \Z_K \cap \prod_{\phi \in \Phi^r} \phi(\a)\Z_L \right],
  \end{split}
\end{align}
we obtain a map from the regular class group of $K^r$ to the Shimura class group of $K$, namely
\begin{equation}\label{def:n_map}
  \begin{split}
    \Nc_{\Phi^r}: \Cl(K^r) & \rightarrow \Cc_K \\
    [\a] & \mapsto ( N_{\Phi^r}(\a), N(\a) ) .
  \end{split}
\end{equation}

\subsection{Torsors and moduli spaces}

We denote by $\Mc_{\Z_K}$ the set of isomorphism classes of ppavs with \emph{primitive} CM by $\Z_K$. Given a primitive CM type $\Phi$, we denote by $\Mc_{\Z_K}(\Phi)$ the set of isomorphism classes of ppavs that admit CM of type $\Phi$.

\begin{proposition}\label{prop:moduli_space_decomposition}
  There is a disjoint union $\Mc_{\Z_K} = \bigcup_{\Phi} \Mc_{\Z_K}(\Phi)$, where $\Phi$ runs over a set of representatives of the equivalence classes of primitive CM types of $K$.
\end{proposition}

\begin{proof}
  This follows from the fact that given a ppav $(A, E)$ with primitive CM by $K$, the CM type of $(A, E)$ is uniquely determined up to equivalence, as reviewed at the beginning of Section \ref{sec:geometric_relevance}.
\end{proof}

Similar to $\Cc_K$, we define the group
\begin{equation}
  \Dc_K = \left \{ (\b, \beta): \b~\text{is fractional $\Z_K$-ideal},~\beta\in K^*~\text{with}~\bb \bbbar = \beta\Z_K \right\}/\sim
\end{equation}
where $(\b, \beta) \sim (\b', \beta')$ if $(\b', \beta') = (x\b, x\xbar \beta)$ for $x\in K^*$. Then there is a tautological injection of groups
\begin{equation}
  \Cc_K \into \Dc_K .
\end{equation}
The pairs $(\bb, \beta)$ representing the elements $[(\bb, \beta)]$ of $\Dc_K$ act on the pairs $(\aa, \xi)$ considered in Section \ref{sec:geometric_relevance} via
\begin{equation}
  (\bb, \beta) (\aa, \xi) = (\bb^{-1} \aa, \beta \xi).
\end{equation}
This action is compatible with the equivalence from Proposition \ref{prop:isom_with_CM} in the sense that if two pairs $(\bb, \beta)$ and $(\bb', \beta')$ are equivalent, then so are $(\bb, \beta) (\aa, \xi)$ and $(\bb', \beta') (\aa, \xi)$.

\begin{proposition}\label{prop:Shimura_ck_action}
  Let $c = [(\bb, \beta)] \in \Dc_K$, and let $(\bb, \beta)(\aa, \xi) = (\aa', \xi')$. Let $\Phi'$ be the CM type of $(\aa', \xi')$. Then $\Phi' = \Phi$ if and only if $c \in \Cc_K$.
\end{proposition}

\begin{proof}
  This follows because the imaginary parts of $\xi$ and $\xi'$ have positive signs at the same complex embeddings if and only if $\beta$ is totally positive.
\end{proof}

Let $c = [(\bb, \beta)] \in \Cc_K$ with $\bb$ integral. Proposition \ref{prop:Shimura_ck_action} shows that the pairs $(\aa, \xi)$ and $c (\aa, \xi)$ have the same CM type $\Phi$, and the inclusion $\Phi (\aa) \subset \Phi (\bb^{-1} \aa)$ yields an isogeny
\begin{equation}
  A (\aa, \xi) \to A (c (\aa, \xi)).
\end{equation}

\begin{proposition}\label{prop:torsor}
  Let $K$ be a sextic CM field, and let $\Phi$ be a fixed primitive CM type. Then the set $\Mc_{\Z_K}(\Phi)$ is a torsor under the action of $\Cc_K$.
\end{proposition}

\begin{proof}
  Proposition \ref{prop:isom_with_CM} shows that $\Mc_{\Z_K} (\Phi)$ is a torsor if it is non-empty. Since $\Mc_{\Z_K} = \bigcup_{\Phi} \Mc_{\Z_K}(\Phi)$, where $\Phi$ runs over the primitive CM types of $K$ up to equivalence, and since the Galois action on the components is transitive by Corollary \ref{cor:trans}, it suffices to prove that $\Mc_{\Z_K} \neq \emptyset$ in order to show that that one, and hence all, of the $\Mc_{\Z_K} (\Phi)$ with $\Phi$ primitive are torsors under $\Cc_K$.

  For this, let $(\aa_0, \xi_0)$ be the pair explicitly constructed in \cite[Proposition 4.4]{streng-algorithm}, which has CM by $\Z_K$. If the CM type $\Phi_0$ associated to $(\aa_0, \xi_0)$ is primitive, then we are done. So suppose that $\Phi_0$ is imprimitive. Then we consider the narrow Hilbert class field $H_0^+$ of $K_0$ and the Hilbert class field $H$ of $K$. The final map in \eqref{eq:exseq} fits into a commutative diagram
  \begin{equation}\label{eq:nmres}
    \begin{tikzcd}
      \Cl (K) \ar[d, "\sim"] \ar[r, "N_{K|K_0}"]& \Cl^+ (K_0) \ar[d, "\sim"] \\
      \Gal (H \ext K) \ar[r, "\text{res}"] & \Gal (H_0^+ \ext K_0)
    \end{tikzcd}
  \end{equation}
  where the map on the bottom is restriction. The inclusion $K H_0^+ \subset H$ yields a surjective restriction map $\Gal (H \ext K) \to \Gal (K H_0^+ \ext K)$. By Galois theory, the latter group is isomorphic to $\Gal (H_0^+ \ext K \cap H_0^+)$. Since $K \cap H_0^+$ is an at most quadratic extension of $K_0$, we see that the image $N$ of $N_{K|K_0}$ is of index at most $2$ in $\Cl^+ (K_0)$.

  Let $\Idf (K_0)$ be the group of fractional $\Z_{K_0}$-ideals. Fix an enumeration of the real embeddings of $K_0$, and given an element $\alpha \in K_0^*$, let $\sgn_i (\alpha)$ denote the sign of $\alpha$ under the $i$th embedding. Then under the map $[ \aa ] \mapsto [(\aa, (1, 1, 1))]$ the narrow class group $\Cl^+ (K_0)$ becomes isomorphic to the group $\Idf (K_0) \times \langle -1 \rangle^3$ modulo the equivalence relation
  \begin{center}
    $(\aa, (s_1, s_2, s_3)) \sim (\aa', (s'_1, s'_2, s'_3))$ \\
    if there exists $\alpha \in K_0^*$ such that $\aa' = \alpha \aa$ and $s'_i = \sgn_i (\alpha) s_i$.
  \end{center}
  The exact sequence
  \begin{equation}
    0 \to \langle -1 \rangle^3 \to \Idf (K_0) \times \langle -1 \rangle^3 \to \Idf (K_0) \to 0
  \end{equation}
  induces another such sequence
  \begin{equation}\label{eq:clexseq}
    0 \to S \to \Cl^+ (K_0) \to \Cl (K_0) \to 0
  \end{equation}
  where $S$ is the quotient of $\langle -1 \rangle^3$ by the image of $\Z_{K_0}^*$ under the sign maps.

  As before, consider the image $N$ of $\Cl (K)$ in $\Cl^+ (K_0)$ under the norm map. We have shown above that $N$ is of index at most $2$ in $\Cl^+ (K_0)$. Moreover, by \cite[Theorem 10.1]{washington}, the norm map $\Cl (K) \to \Cl (K_0)$ is surjective. So if we identify $S$ with its image in $\Cl^+ (K_0)$, then \ref{eq:clexseq} shows that $N \cap S$ is of index at most $2$ in $S$. This means that there exists a $\Z_K$-ideal $\bb$ such that $\bb \bbbar$ is principal, and moreover generated by an element $\beta \in K_0$ whose signs at the infinite places of $K_0$ do not all coincide.

  Now let $(\aa, \xi) = (\bb, \beta) (\aa_0, \xi_0)$. Then because of the sign property of $\beta$, the CM type $\Phi$ corresponding to $(\aa, \xi)$ differs from both $\Phi_0$ and $\overline{\Phi}_0$. Our classification of the CM types of sextic CM fields in Section \ref{sec:cm} then shows that $\Phi$ is primitive. Therefore $\Mc_{\Z_K}$ is non-empty, since it contains the ppav that corresponds to $(\aa, \xi)$.
\end{proof}

\subsection{Representatives up to Galois conjugation}

Let us fix a primitive CM type $\Phi$ of $K$. Then the set $\Mc_K(\Phi)$ is stable under the action of $G^r = \Gal (\Qbar \ext K^r)$, and the orbits of $\Mc_{\Z_K}(\Phi)$ under the action of the group $G^r$ correspond to the elements of the quotient $\Cc_K / \im (\Nc_{\Phi^r})$. More precisely, we have the following result.

\begin{theorem}[Main Theorem of complex multiplication]\label{thm:srecip}
  Let $A (\aa, \xi)\in \Mc_{\Z_K}(\Phi)$, and $\sigma \in G^r$. Denote by $\bb\in \Cl(K^r)$ the ideal whose class corresponds to $\sigma$ under the Artin map. Then
  \begin{equation}\label{eq:srecip}
    \sigma (A (\aa, \xi)) \cong A (\Nc_{\Phi^r} (\b) (\aa, \xi)) .
  \end{equation}
\end{theorem}

We will use this reciprocity law to prove Theorem \ref{thm:boundexp}, which shows that given $A (\aa_0, \xi_0)$ in $\Mc_{\Z_K} (\Phi)$, we can obtain any other isomorphism class $A (\aa, \xi)$ in $\Mc_{\Z_K} (\Phi)$ as a Galois conjugate of an abelian variety $A (\bb \aa_0, \eta_0)$, where $\bb$ runs through a fixed (small) set of representatives of $G_2 / e G_2$, where $G_2 = \ker (N_{K | K_0}) \subset \Cl (K)$ and where $e | 4$. The usefulness of this result stems from the fact that $G_2 / e G_2$ is usually far smaller than $G_2$ itself. We start with a general observation.

\begin{proposition}\label{prop:indep}
  Let $\Phi$ and $\Psi$ be Galois equivalent CM types. Then there is an equality of double reflex maps
  \begin{equation}
    N_{\Phi^r} \circ N_\Phi = N_{\Psi^r} \circ N_\Psi .
  \end{equation}
\end{proposition}

\begin{proof}
  Let $\Psi = \sigma \Phi$ for $\sigma \in \Gal (L \ext \Q)$. Then we have
  \begin{equation}
    N_\Psi (\aa)
    = \prod_{\psi \in \Psi} \psi (\aa)
    = \prod_{\phi \in \Phi} \sigma (\phi (\aa))
    = \sigma (\prod_{\phi \in \Phi} \phi (\aa))
    = \sigma (N_\Phi (\aa))
  \end{equation}
  for all ideals $\aa$ of $\Z_K$. Moreover, by Proposition \ref{prop:conjtype} we have
  \begin{equation}
    N_{\Psi^r} (\bb)
    = \prod_{\psi \in \Psi^r} \psi (\bb)
    = \prod_{\phi \in \Phi^r} \phi (\sigma^{-1} (\bb))
  \end{equation}
  for all ideals $\bb$ of $\Z_{K^r}$. Therefore
  \begin{equation}
    N_{\Psi^r} (N_\Psi (\aa))
    = N_{\Psi^r} (\sigma (N_\Phi (\aa)))
    = \prod_{\phi \in \Phi^r} \phi (\sigma^{-1} (\sigma (N_\Phi (\aa))))
    = \prod_{\phi \in \Phi^r} \phi (N_\Phi (\aa))
    = N_{\Phi^r} (N_\Phi (\aa))
  \end{equation}
  for all ideals $\aa$ of $\Z_K$, which proves our claim.
\end{proof}

For analogs of the next statement in the quartic case, see \cite[Proof of Theorem III.2.2]{streng-thesis} and \cite[Proposition 2.3.1]{kilicer3}: These results in turn go back to \cite[Equality (3.1)]{shimura-onavs}.

\begin{lemma}\label{lemma:relations_c6_d6}
  Let $K$ be a sextic CM field with Galois group isomorphic to $C_6$ or $D_6$, and let $\Psi$ be the distinguished CM type of $K$. Then for all fractional $\Z_K$-ideals $\aa$ we have an equality of fractional $\Z_L$-ideals
  \begin{equation}
    N_{\Phi^r} (N_\Phi(\aa)) =  N_{K|\Q}(\aa) \aa \aabar^{-1} N_{\Psi} (\aa).
  \end{equation}
  If $\Phi$ is imprimitive, then $N_{\Phi^r} (N_\Phi(\a)) = N_\Phi(\a)$.
\end{lemma}

\begin{proof}
  We prove the statement for the case where the Galois group is isomorphic to $D_6$. The Galois case is similar. Using the notation in Proposition \ref{prop:d6}, let $H = \Gal (L \ext K) = \langle \tau \rangle$, and let $\sigma$ be the generator of the set of embeddings of $K$ into $L$. By Proposition \ref{prop:indep}, it suffices to consider the case $\Phi_L = \{0, 1, 2\}$. Then $\Phi_L^r = \{0, 5, 4\}$, and the double norm computation gives rise to the element
  \begin{equation}
    (1 + \sigma^5 + \sigma^4) (1 + \sigma + \sigma^2) = 3 + 2\sigma + \sigma^2 + \sigma^4 + 2\sigma^5
  \end{equation}
  in the group algebra of $\Gal (L \ext \Q)$. If we consider the elements in this sum up to right multiplication by elements of the group $H$, we get
  \begin{equation}
    \begin{aligned}
      3H + 2\sigma H + \sigma^2 H + \sigma^4 H + 2\sigma^5 H & = (H + \sigma H + \sigma^2 H + \sigma^3 H + \sigma^4 H + \sigma^5 H)\\ & + (H - \sigma^3 H)  + (H + \sigma H  +  \sigma^5H).
    \end{aligned}
  \end{equation}
  The first two terms correspond to $N_{K|\Q}(\aa)$ and $\aa \aabar^{-1}$, respectively. The last sum corresponds to the CM type $\Psi_L = \{0, 1, 5\}$, and is independent of the choice of $\Phi_L$.

  In the imprimitive case we can take $\Phi_L = \{0, 2, 4\}$. The reflex field is then the unique quadratic CM subfield of $K$, and the reflex type its canonical inclusion, which shows our claim.
\end{proof}

\begin{proposition}\label{prop:bb_square_image_rtypenorm}
  Let $K$ be a sextic CM field with Galois group isomorphic to $C_6$ or $D_6$, and let $\Phi$ be a primitive CM type of $K$. If $[\b] \in \Cl(K)$ satisfies $\b\bbbar = \beta\Z_K$ for $\beta \in K_0^*$, then $[ \b^2 ]$ is in the image of the reflex type norm $N_{\Phi^r} : \Cl (K^r) \to \Cl (K)$.
\end{proposition}

\begin{proof}
  Let $\Psi$ be the distinguished primitive CM type of $K$ from Definition \ref{def:dist}. If $\a = N_{\Psi}(\b)\Z_L$, then by Lemma \ref{lemma:relations_c6_d6} we have an equality of fractional $\Z_L$-ideals
  \begin{equation}\label{eq:b_square_image_rtn}
    \begin{split}
         N_{\Psi^r}(\a) = N_{\Psi^r}(N_{\Psi} (\b))
       = N_{K|\Q}(\b)  \b\bbbar^{-1}  N_\Psi(\b)
       = N_{K|\Q}(\b) \beta^{-1}  N_{\Psi^r}(\b) \b^2.
    \end{split}
  \end{equation}
  (Here we have used the fact that the reflex type of $(K, \Psi)$ is given by $(K, \Psi)$, as was shown in the proof of Proposition \ref{prop:reflex_d6}.) We therefore have that $N_{\Psi^r} ([\aa]) = N_{\Psi^r} ([ \b ]) [ \b^2 ]$ and $[ \bb^2 ] = N_{\Psi^r} ([ \a \b^{-1} ]) \in \im (N_{\Psi^r})$. But Proposition \ref{prop:conjtype} implies that if $\Phi = \sigma (\Psi)$, then $\Phi^r = \Psi^r \sigma^{-1}$, hence $N_{\Phi^r}$ and $N_{\Psi^r}$ have equal images in $\Cl(K)$. (Indeed, if $\bb = N_{\Psi^r} (\cc)$, then $\bb = N_{\Phi^r} (\sigma (\cc))$.) Since all primitive CM types are Galois equivalent, we obtain our claim.
\end{proof}

\begin{proposition}
  Let $K$ be a sextic CM field with Galois group isomorphic to $C_6$ or $D_6$ with distinguished CM type $\Psi$. For any primitive CM type $\Phi$ of $K$ and any equivalence class $(\b, \beta)$ in $\Cc_K$ the equivalence class of $( N_{\Psi}(\b) \b^2, N(\b) \beta^2 )$ is in the image of the map $ \Nc_{\Phi^r} : \Cl(K^r) \rightarrow \Cc_K$. Furthermore, if $N_{\Psi}(\b) = \mu\Z_K$ is a principal $\Z_K$-ideal, then said image $(N_{\Psi}(\b) \b^2, N(\b) \beta^2)$ is equivalent to $( \b^2, \beta^2 )$.
\end{proposition}

\begin{proof}
  With $\a = N_{\Phi}(\b)$ and Lemma~\ref{lemma:relations_c6_d6} we get that
  \begin{equation}
    \begin{aligned}
      \Nc_{\Phi^r} (\a) & = ( N_{\Phi^r}(\a), N(\a) ) = ( N_{\Phi^r}(N_{\Phi} (\b)), N(N_{\Phi} (\b)) )\\
      & = ( N(\b)_{K|\Q}  \beta^{-1} N_{\Psi}(\b) \b^2 , N(\b)_{K|\Q}  \beta^{-1} N_{\Phi}(\b) \b^2 \overline{N(\b)_{K|\Q}  \beta^{-1} N_{\Phi}(\b) \b^2} ) \\ & = ( N(\b)_{K|\Q} \beta^{-1} N_{\Psi}(\b) \b^2, N(\b)_{K|\Q}^2 N_{\Phi}(\b) \overline{N_{\Phi}(\b)} ) \\ &  = ( N(\b)_{K|\Q} \beta^{-1} N_{\Psi}(\b) \b^2, N(\b)_{K|\Q}^3) .
    \end{aligned}
  \end{equation}
  Note that $N_{\Psi}(\b)$ is indeed an ideal of $\ZZ_K$ since $(K, \Psi)$ is its own reflex, and that we indeed have that
  \begin{equation}
    \begin{aligned}
      N_{K|K_0}( N_{K|\Q}(\b)  \beta^{-1} N_{\Psi}(\b) \b^2 ) & = N_{K|\Q}(\b)  \beta^{-1} N_{\Psi}(\b) \b^2 \overline{ N_{K|\Q}(\b)  \beta^{-1} N_{\Psi}(\b) \b^2 }\\ & = N_{K|\Q}(\b)^2 (\b\bbbar)^{-1} \overline{ (\b\bbbar)^{-1}} N_{\Psi}(\b) \overline{N_{\Psi}(\b)} \b^2 \bbbar^2 \\ & = N_{K|\Q}(\b)^3 \Z_K.
    \end{aligned}
  \end{equation}
  Then since $\beta \in K_0$, the equivalence relation \eqref{def:shimura_class_group} yields
  \begin{equation}
    \begin{aligned}
      ( N_{K|\Q}(\b)  \beta^{-1} N_\Psi(\b) \b^2, N_{K|\Q}(\b)^3 ) & \sim ( \beta^{-1} N_\Psi(\b)  \b^2, N_{K|\Q}(\b) ) \sim ( N_{\Psi}(\b) \b^2, N_{K|\Q}(\b) \beta^2 ),
    \end{aligned}
  \end{equation}
  This shows the first claim. If $N_{\Psi}(\b) = \mu\Z_K$ is a principal ideal, then
  \begin{eqnarray*}
    ( N_{\Psi}(\b) \b^2, N_{K|\Q}(\b) \beta^2 ) \sim ( \mu \b^2, N_{K|\Q}(\b) \beta^2 ) \sim ( \b^2, (\mu \overline{\mu})^{-1} N_{K|\Q}(\b) \beta^2 ) \sim (\b^2,\beta^2)
  \end{eqnarray*}
  which shows the second claim.
\end{proof}

\begin{lemma}\label{lemma:relations_sa_ss}
  Let $K$ be a sextic CM field with Galois group isomorphic to $C_2^3\rtimes C_3$ or $C_2^3 \rtimes S_3$, and let $\Phi$ be a CM type of $K$. Then for all fractional $\Z_K$-ideals $\aa$ we have an equality of fractional $\Z_L$-ideals
  \begin{equation}
    N_{\Phi^r} (N_\Phi(\aa)) = N_{K|\Q} (\aa)^2 (\aa \aabar^{-1})^2 .
  \end{equation}
\end{lemma}

\begin{proof}
  We prove this for the CM type $\Phi_1$ and Galois group $C_2^3\rtimes C_3$: The statement for the other CM types follows from Proposition \ref{prop:indep}, and the argument for the group $C_2^3\rtimes S_3$ is similar. Using the notation in Proposition \ref{prop:reflex_c23_c3_s3}, the extensions of $\Phi_1$ and $\Phi_1^r$ to $L$ are given by $\left\{ 1, \sigma, \sigma^2 \right\}$ and $\left\{ 1, n_1, n_2, n_1 n_2 \right\}$, respectively. Considering the given double norm comes down to studying the element
  \begin{equation}
    (1 + n_1 + n_2 + n_1 n_2) (1 + \sigma + \sigma^2) = 1 + n_1 + n_2 + n_1 n_2 + \sigma + n_1 \sigma + n_2 \sigma + n_1 n_2 \sigma + \sigma^2 + n_1 \sigma^2 + n_2 \sigma^2 + n_1 n_2 \sigma^2
  \end{equation}
  in the group algebra of $\Gal (L \ext \Q)$, where we consider the elements in this sum up to right multiplication by elements of the subgroup $H = \langle n_1, n_2 \rangle$ that corresponds to the field $K$. In terms of the cosets in \eqref{eq:cosets}, this yields
  \begin{equation}
    \begin{split}
      & H + n_1 H + n_2 H + n_1 n_2 H + \sigma H + n_1 \sigma H + n_2 \sigma H + n_1 n_2 \sigma H + \sigma^2 H + n_1 \sigma^2 H + n_2 \sigma^2 H + n_1 n_2 \sigma^2 H \\
      = & H + H + H + H + \sigma H + \sigma \rho H + \sigma H + \sigma \rho H + \sigma^2 H + \sigma^2 H + \sigma^2 \rho H + \sigma^2 \rho H \\
      = & 4 H + 2 \sigma H + 2 \sigma^2 H + 2 \sigma \rho H + 2 \sigma^2 \rho H \\
      = & (2 H + 2 \sigma H + 2 \sigma^2 H + 2 \rho H + 2 \sigma \rho H + 2 \sigma^2 \rho H) + (2 H - 2 \rho H),
    \end{split}
  \end{equation}
  which shows the claim.
\end{proof}

\begin{proposition}\label{prop:bb_dsquare_image_rtypenorm}
  Let $K$ be a sextic CM field with Galois group isomorphic to $C_2^3\rtimes C_3$ or $C_2^3 \rtimes S_3$, and let $\Phi$ be a CM type of $K$. If $[\b] \in \Cl(K)$ satisfies $\b\bbbar = \beta\Z_K$ for $\beta \in K_0^*$, then $[ \b^4 ]$ is in the image of the reflex type norm $N_{\Phi^r} : \Cl (K^r) \to \Cl (K)$.
\end{proposition}

\begin{proof}
  Once more we only give the proof for the Galois group $C_2^3\rtimes C_3$. If $\a = N_{\Phi}(\b)\Z_L$, then by Lemma \ref{lemma:relations_sa_ss} we have an equality of fractional $\Z_L$-ideals
  \begin{equation}\label{eq:b_dsquare_image_rtn}
    \begin{split}
         N_{\Phi^r}(\a)   = N_{\Phi^r}(N_{\Phi} (\b)) =  N_{K|\Q}(\b)^2  \bigg( \b\bbbar^{-1} \bigg)^2  =  N_{K|\Q}(\b)^2  \beta^{-2}  \b^4.
    \end{split}
  \end{equation}
  We therefore have that $N_{\Phi^r} ([\aa]) = [ \b^4 ]$, which shows the claim.
\end{proof}

\begin{proposition}
  Let $K$ be a sextic CM field with Galois group isomorphic to $C_2^3\rtimes C_3$ or $C_2^3 \rtimes S_3$. For any CM type $\Phi$ of $K$ and any equivalence class $(\b, \beta)$ in $\Cc_K$, the equivalence class of $( \b^4, \beta^4 )$ is in the image of the map $ \Nc_{\Phi^r}: \Cl(K^r) \rightarrow \Cc_K$.
\end{proposition}

\begin{proof}
  With $\a = N_{\Phi}(\b)$ and Lemma~\ref{lemma:relations_sa_ss} we get
  \begin{equation}
    \begin{aligned}
      \Nc_{\Phi^r} (\a) & = ( N_{\Phi^r}(\a), N(\a) ) = ( N_{\Phi^r}(N_{\Phi} (\b)), N(N_{\Phi} (\b)) )\\ & = ( N_{K|\Q}(\b)^2  \beta^{-2}  \b^4, N_{K|\Q}(\b)^2 (\b\bbbar^{-1})^2 \overline{N_{K|\Q}(\b)^2 (\b\bbbar^{-1})^2} ) \\ & = ( N_{K|\Q}(\b)^2  \beta^{-2}  \b^4, N_{K|\Q}(\b)^4 ),
    \end{aligned}
  \end{equation}
  Since $\beta \in K_0$, using the equivalence relation on the Shimura class group yields that
  \begin{equation}
    ( N_{K|\Q}(\b)^2  \beta^{-2}  \b^4, N_{K|\Q}(\b)^4 ) \sim ( \beta^{-2}\b^4, 1 )\sim (\b^4, \beta^4),
  \end{equation}
  which shows the claim.
\end{proof}

We can now state the main result of this section:

\begin{theorem}\label{thm:boundexp}
  Let $G_2 = \ker (N_{K | K_0}) \subset \Cl (K)$ be the subgroup of classes $[ \bb ]$ with the property that $\bb \bbbar$ is generated by a totally positive element of $K_0$. Let $B$ be a set of ideals that yields representatives of the quotient $Q = G_2 / e G_2$, where $e = 2$ if $\Gal (K) \in \left\{ C_6, D_6 \right\}$ and where $e = 4$ if $\Gal (K) \in \left\{ C_2^3 \rtimes C_3, C_2^3 \rtimes S_3 \right\}$. Similarly, let $V$ be a set of units that yields representatives of the quotient $(\Z^*_{K_0})^+ / N_{K | K_0} (\Z_K^*)$.

  Fix $A (\aa_0, \xi_0)$ in $\Mc_{\Z_K} (\Phi)$, and let $A (\aa, \xi)$ in $\Mc_{\Z_K} (\Phi)$ be given. Then the Galois orbit of $A (\aa, \xi)$ under the action of $G^r = \Gal (\Qbar \ext K^r)$ contains an abelian variety isomorphic to $A (\bb \aa_0, v \beta^{-1} \xi_0)$, where $\bb \in B$, where $\beta \in K_0$ generates $\bb \bbbar$, and where $v \in V$.
\end{theorem}

\begin{proof}
  Proposition~\ref{prop:torsor} along with the exact sequence in Equation~\eqref{eq:exseq} shows that $A (\aa, \xi)$ is isomorphic to $A (\bb \aa_0, v \beta^{-1} \xi_0)$ for some ideal $\bb$ with $[ \bb ] \in G_2$ and for some $v \in V$. Propositions \ref{prop:bb_square_image_rtypenorm} and \ref{prop:bb_dsquare_image_rtypenorm} show that the first component of the map $\Nc_{\Phi^r}$ surjects onto $e G_2$. Applying a corresponding Galois conjugation to $A (\aa, \xi)$ if needed, we may therefore assume that $\bb \in B$, after which another invocation of Equation~\eqref{eq:exseq} shows our claim.
\end{proof}

\section{Running through the LMFDB}\label{sec:lmfdb}

\subsection{Algorithms}\label{sec:algs}

Let $K$ be a sextic CM field. The considerations in the previous sections give rise to a method to determine representatives of the set of ppavs with CM by $\Z_K$ up to isomorphism and Galois conjugation. We split up the steps of this method into several algorithms. Throughout, we fix not only the CM field $K$ but also a primitive CM type $\Phi$ of $K$ with values in $\C$. (It is in fact not essential that $\Phi$ be primitive, but this is the case that interests us in the current article.) Similar algorithms were considered in lower genus in the previous works \cite{engethome} and \cite{streng-algorithm}. We discuss differences in our approach as we go.

\renewcommand{\theenumi}{\arabic{enumi}}
\begin{algorithm}\label{alg:algo1}(Precomputation step)

  \textsc{Input:} A sextic CM field $K$.

  \textsc{Output:} Precomputed data used in the later Algorithms \ref{alg:algo2}, \ref{alg:algo3}, and \ref{alg:algo4}.

  \begin{enumerate}
    \item Determine the class group and unit group $\Cl (K), \Z_K^*$ of $K$ and the class group, narrow class group and unit group $\Cl (K_0), \Cl^+ (K_0), \Z_{K_0}^*$ of its totally real subfield $K_0$;
    \item Determine the subgroup $G_1 \subset \Cl (K)$ of classes $[ \aa ] \in \Cl (K)$ with the property that $\aa \aabar$ is generated by an element of $K_0$;
    \item Determine the subgroup $G_2 \subset \Cl (K)$ of classes $[ \aa ] \in G_1$ with the property that $\aa \aabar$ is generated by a totally positive element of $K_0$;
    \item Let $Q = G_2 / e G_2$, where $e = 2$ if $\Gal (K) \in \left\{ C_6, D_6 \right\}$ and where $e = 4$ if $\Gal (K) \in \left\{ C_2^3 \rtimes C_3, C_2^3 \rtimes S_3 \right\}$;
    \item Determine a set of ideals $C$ of $\Z_K$ that yields representatives of the quotient $G_1 / G_2$, as well as a set of ideals $B$ of $\Z_K$ that yields representatives of the quotient $Q = G_2 / e G_2$;
    \item Determine the subgroup $U_1 \subset \Z_{K_0}^*$ of totally positive units in $\Z_{K_0}^*$;
    \item Determine the subgroup $U_2 \subset U_1$ of units in $\Z_{K_0}^*$ that are norms from $\Z_K^*$;
    \item Determine a set of units $W$ that yields representatives of the quotient $\Z_{K_0}^* / U_1$, as well as a set of units $V$ that yields representatives of the quotient $U_1 / U_2$.
  \end{enumerate}
\end{algorithm}
\renewcommand{\theenumi}{\roman{enumi}}

The steps in this algorithm can be performed by using classical algorithms for class and unit groups. We only give additional remarks on steps that are somewhat less standard.

\begin{remark}\label{rem:algo1}
  \hspace{0cm}
  \begin{enumerate}
    \item Under the generalized Riemann hypothesis, the calculation of the class and unit group of $K$ and $K_0$ in Step (1) speeds up tremendously. We have therefore used this assumption while performing our calculations.

    \item We can determine the subgroup $G_1$ in Step (2) as the kernel of the homomorphism $\Cl (K) \to \Cl (K_0)$ given by $[ \aa ] \mapsto [ \aa \aabar ]$, and $G_2$ as the kernel of a similar homomorphism to $\Cl^+ (K_0)$. Similar considerations apply to the determination of $U_1$ and $U_2$ in Steps (6) and (7).

    \item It is important that the representatives returned by Algorithm \ref{alg:algo1} be minimized, since otherwise large precision loss will occur in later steps. In \cite[\S 4.1]{engethome}, this minimization is also mentioned as being useful when working with the Shimura class group in the genus-2 case. For our purposes this is not merely useful, but also crucial in practice, as the class groups involved are of considerable size and working with large powers of ideal class generators without reducing these already causes unacceptable precision loss when determining the corresponding lattices in $\C^3$ in Algorithm \ref{alg:algo4}. We therefore spend a few lines on this reduction step.

      For an ideal representative in $B$ and $C$, this observation means that it should be multiplied with a principal ideal in such a way that the norm of the resulting product is smaller than the Minkowski bound $M$ of $K$. This can be done as follows. Given an ideal $\aa$ to be minimized, one computes the lattice $\Gamma$ in $\CC^3$ that is the image of $\aa^{-1}$ under the complex embeddings of $K$. One then determines a short vector $\alpha$ in $\Gamma$, and the corresponding element $\alpha$ of $\aa^{-1}$ will satisfy $N_{K | \Q} (\alpha) \le M N_{K | \Q} (\aa^{-1})$. Therefore the norm of the ideal $\alpha \aa$ is at most $M$, and we use this product as a minimized ideal representative.

      For a unit that yields a representative in $V$, resp.\  $W$, being small means the following. Let $\ell : \Z_{K_0}^* \to \R^2$ be the log map whose image is the Dirichlet lattice of the unit group $\Z_{K_0}^*$. Then given an element $u$ of $V$ (resp.\  $W$) to be minimized, we can use closest vector algorithms to find an element $u_1$ (resp.\  $u_2)$ of $U_1$ (resp.\  $U_2$) such that that $\ell (u) + \ell (u_1)$ (resp.\ $\ell (u) + \ell (u_2)$) is small, and we use the corresponding product $u \cdot u_1$ (resp.\ $u \cdot u_2$) as a minimized unit representative.

    \item Note that in contrast to the methods in \cite{engethome}, our precomputation does not require the computation of the Shimura class group or the image of the reflex norm, which simplifies its description.
  \end{enumerate}
\end{remark}

\renewcommand{\theenumi}{\arabic{enumi}}
\begin{algorithm}\label{alg:algo2}(Determining an initial triple $(\Phi, \aa, \xi)$)

  \textsc{Input:} A sextic CM field $K$ and a primitive CM type $\Phi$ of $K$.

  \textsc{Output:} A single triple $(\Phi, \aa, \xi)$, with $\aa$ a fractional $\Z_K$-ideal and with $\xi \in K$ totally imaginary, such that $(\Phi, \aa, \xi)$ represents a principally polarized abelian threefold $A$ with CM by $K$ of $\Phi$.

  \begin{enumerate}
    \item Determine a pair $(\aa_0, \xi_0)$ such that $(\xi_0) = (\aa_0 \overline{\aa_0} \Df_{K | \Q})^{-1}$. If the imaginary part of $\xi_0$ is positive for all embeddings in $\Phi$, then return $(\Phi, \aa_0, \xi_0)$. Otherwise, proceed to the next step.
    \item Run through the elements $\cc$ of $C$, and let $\gamma \in K_0$ be a generator of $\cc \overline{\cc}$.
    \item Within the previous loop, run through the elements $w$ of $W$, and consider $(\aa, \xi) = (\cc \aa_0, w \gamma^{-1} \xi_0)$. If $(\aa, \xi)$ admits $\Phi$ as a CM type, or in other words, if $\xi$ has positive imaginary part for the embeddings in $\Phi$, then return $(\Phi, \aa, \xi)$.
  \end{enumerate}
\end{algorithm}
\renewcommand{\theenumi}{\roman{enumi}}

\begin{proof}
  If the algorithm returns a triple, then it is correct by construction. It therefore remains to show that the algorithm does always give an output.

  First note that the existence of a triple $(\Phi, \aa, \xi)$ as in the Output step follows from Proposition \ref{prop:torsor}. Now suppose that we have determined a pair $(\aa_0, \xi_0)$ as in Step (1) of the algorithm. Then since both $\aa \aabar \Df_{K | \Q}^{-1}$ and $\aa_0 \aabar_0 \Df_{K | \Q}^{-1}$ are principal, and generated by totally imaginary elements of $K$, we have that the class of $\aa \aa_0^{-1}$ belongs to $G_1$. Let $\cc \in C$ be an element representing this class, and let $\gamma \in K_0$ be the chosen generator of $\cc \overline{\cc}$. We can then write $\aa = \delta \cc \aa_0$ with $\delta \in K^*$. Let $\bb = \cc \aa_0$. Then
  \begin{equation}
    (\delta \overline{\delta} \xi) = ((\delta \overline{\delta})^{-1} \aa \aabar \Df_{K | \Q})^{-1} = (\bb \bbbar \Df_{K | \Q})^{-1}
  \end{equation}
  and
  \begin{equation}
    (\gamma^{-1} \xi_0) = ((\cc \ccbar) \aa_0 \aabar_0 \Df_{K | \Q})^{-1} = (\bb \bbbar \Df_{K | \Q})^{-1},
  \end{equation}
  so since $\xi$ and $\xi_0$ are totally imaginary, we have $\delta \overline{\delta} \xi = u \gamma^{-1} \xi_0$ for a unit $u \in \ZZ_{K_0}^*$. Let $w \in W$ be a representative of the class corresponding to $u$. Then $(\cc \aa_0, w \gamma^{-1} \xi_0)$ has the property that the imaginary parts of $w \gamma^{-1} \xi_0$ has the same signs as $\delta \overline{\delta} \xi$, and hence as $\xi$. These are exactly the signs compatible with $\Phi$. Therefore since the algorithm encounters this triple as it runs, it is indeed guaranteed to return the requested output.
\end{proof}

\begin{remark}
  Finding a pair $(\aa_0, \xi_0)$ as in Step (1) of Algorithm \ref{alg:algo2} is possible by using the methods of \cite[Proposition 4.4]{streng-algorithm}: In fact the pair $(\aa_0, y z)$ in \loccit\ can be used.
\end{remark}

Given an initial triple $(\Phi, \aa, \xi)$ returned by Algorithm \ref{alg:algo2}, a full set of such triples (in the sense of Theorem \ref{thm:boundexp}) can be determined quickly by using the precomputed data from Algorithm \ref{alg:algo1}:

\renewcommand{\theenumi}{\arabic{enumi}}
\begin{algorithm}\label{alg:algo3}(Determining all triples $(\Phi, \aa, \xi)$)

  \textsc{Input:} A sextic CM field $K$ and a primitive CM type $\Phi$ of $K$.

  \textsc{Output:} A set $S$ of triples $(\Phi, \aa, \xi)$ as in Section \ref{sec:geometric_relevance}, so that $(\aa, \xi)$ represents a principally polarized abelian threefold $A$ that admits CM by $K$ of $\Phi$. Moreover, $S$ satisfies the following property: Up to Galois conjugation over the reflex field $K^r$, any pair $(\Phi, A)$, where $A$ is a principally polarized abelian threefold that admits CM by $\Z_K$, is isomorphic over $\C$ to an abelian variety corresponding to one of the elements of $S$.

  \begin{enumerate}
    \item Let $(\Phi, \aa_0, \xi_0)$ be the triple from Algorithm \ref{alg:algo2}.
    \item Run through the elements $\bb$ of $B$, and let $\beta \in K_0$ be a generator of $\bb \overline{\bb}$.
    \item Within the previous loop, run through the elements $v$ of $V$, and add $(\aa, \xi) = (\bb \aa_0, v \beta^{-1} \xi_0)$ to $S$.
    \item Return $S$ once the loops above have terminated.
  \end{enumerate}
\end{algorithm}
\renewcommand{\theenumi}{\roman{enumi}}

\begin{proof}
  The correctness of Algorithm \ref{alg:algo3} follows from Theorem \ref{thm:boundexp}.
\end{proof}

\begin{remark}\label{rem:factor}
  \hspace{0cm}
  \begin{enumerate}
    \item We do not claim that the given set $S$ is in actual bijection with the set of isomorphism classes of pairs $(\Phi, A)$ up to Galois conjugation, and indeed this is not the case in general. For our purposes it is enough to ensure that we obtain at least one triple $(\Phi, \aa, \xi)$ for each isomorphism class up to Galois conjugation, and we do not impose additionally that we have only a single triple for each Galois conjugacy class.

    \item As was shown in Equation \eqref{eq:strengsize}, the size of $\Mc_{\Z_K}(\Phi)$ is well approximated by $h (K) / h^+ (K_0)$. When $G = C_2^3 \rtimes S_3$, the largest size of the quotient $h (K) / h^+ (K_0)$ in \eqref{eq:strengsize} was $11287$, whereas the largest size of the group $Q = G_2 / e G_2$ from Theorem \ref{thm:boundexp} was $128$, thus showing the speed gain that our taking into account of Galois conjugacy provides.
  \end{enumerate}
\end{remark}

\renewcommand{\theenumi}{\arabic{enumi}}
\begin{algorithm}\label{alg:algo4}(Determining period matrices)

  \textsc{Input:} A sextic CM field $K$ and a primitive CM type $\Phi$ of $K$.

  \textsc{Output:} The small period matrices $\tau$ corresponding to the elements of the set $S$ in Algorithm \ref{alg:algo3}, sorted into two sets $T_H$ and $T_N$ that (heuristically) give rise to hyperelliptic and non-hyperelliptic curves, respectively.

  \begin{enumerate}
    \item Determine the set $S$ from Algorithm \ref{alg:algo3}, and initialize $T_H$ and $T_N$ to be empty sets.
    \item Let $(\Phi, \aa, \xi)$ be in $S$. Compute the corresponding principally polarized abelian threefold $(A, E)$ in the usual manner \cite[\S 4]{streng-algorithm}, setting $A = \CC^3 / \Phi (\aa)$ and letting $E$ be the $\R$-linear extension of the trace pairing $(\alpha, \beta) \mapsto \Tr_{K | \Q} (\xi \alpha \overline{\beta})$.
    \item Determine a Frobenius alternating form of $E$ to find some big period matrix $P \in M_{3,6} (\C)$ for $A$, and from it, a small period matrix $\tau \in M_{3,3} (\C)$.
    \item Reduce $\tau$ by using the methods from \cite[\S 2]{KLRRSS}.
    \item Use algorithms, for example those by Labrande \cite{labrande}, to determine whether $\tau$ has $1$ or $0$ vanishing even theta-null values to some high precision (typically 100 digits). In the former case, add $\tau$ to $T_H$; in the latter, add it to $T_N$.
  \end{enumerate}
\end{algorithm}
\renewcommand{\theenumi}{\roman{enumi}}

\begin{remark}
  \hspace{0cm}
  \begin{enumerate}
    \item Note that our algorithms differ from those in \cite{streng-algorithm}, as we fix our primitive CM type $\Phi$ throughout. When considering CM curves up to Galois conjugation, we are justified in doing so because of Corollary \ref{cor:trans}.

    \item Because we have ensured that $\Phi$ is a primitive CM type, the associated abelian threefolds are indeed Jacobians of genus-$3$ curves. The criterion for said Jacobian to be hyperelliptic in terms of even theta-null values is \cite[Lemmata 10 and 11]{igusa}.

    \item Like the minimization of representatives in Algorithm \ref{alg:algo1}, Step (4) of Algorithm \ref{alg:algo4} is essential to keep its running time short.

    \item Our own run of Algorithm \ref{alg:algo4} used the native \Magma\ implementation in Step (5) instead of the algorithms from \cite{labrande}. The even theta values were computed to 100 digits of precision, and decided to be numerically equal to zero when their absolute value is at most $10^{-50}$. Setting \texttt{Labrande := true} in the implementation at \cite{dis-github} allows for an alternative verification of these results using \cite{labrande} instead.

    \item Using interval arithmetic or the fast decay of the terms appearing in the sum that define an even theta-null value, it is in principle possible to verify rigorously whether such a value is non-zero, as it suffices to check that the sum of its initial terms is of sufficiently large absolute value. This allows one to prove that a ppav $A$ that Algorithm \ref{alg:algo4} suspects to be a non-hyperelliptic Jacobian is indeed such a Jacobian. By contrast, showing that $A$ comes from a hyperelliptic curve is more involved. For the moment, we see no other rigorous method to check this than to compute an equation for a corresponding curve $X$ as in Section \ref{sec:equations} and to show that $\Jac (X)$ has CM using the algorithms in \cite{cmsv-endos}. We discuss some further sanity checks in the next section.
  \end{enumerate}
\end{remark}

\subsection{Fields}\label{sec:fields}

With the algorithms from Section \ref{sec:algs} in hand, we considered the sextic CM fields in the LMFDB \cite{lmfdb}. We have applied our algorithms, implemented in \Magma\ \cite{magma} and available at \cite{dis-github}, to all of these $547,156$ fields, except for 2 fields with Galois group $D_6$ whose root discriminant exceeds $10^{12}$; for these, the calculation of the class and unit group did not finish in a timely fashion even when assuming the generalized Riemann hypothesis. Note that the list from the LMFDB contains the complete list of sextic CM fields of absolute discriminant at most $10^7$. The total computation required $4$ days on $20$ cores when working to relatively high precision to exclude rounding errors.

We list our results, which imply Main Results 1 and 2, in Table \ref{tab:results}. The first column of this table lists the possible Galois groups, which are as in Theorem \ref{thm:sexticCM}. Given such a group in the first column, the second column of the table indicates the number of CM fields $K$ in the database whose Galois group is isomorphic to the specified group. The third column indicates the number of such CM fields $K$ for which the set $T_H$ from Algorithm \ref{alg:algo4} is non-empty, or in other words the number of such CM fields $K$ for which there (heuristically) exists a hyperelliptic curve whose Jacobian has \emph{primitive} CM by $\Z_K$. For simplicity, we call such a CM field $K$ \defi{hyperelliptic}.

As was already known, and as can be deduced from the classification in \cite{LLRS}, if a CM field $K$ contains $\Q (i)$, then any curve whose Jacobian has primitive CM by $\Z_K$ is automatically hyperelliptic. We call a hyperelliptic CM field $K$ that does \emph{not} include $\Q (i)$ \defi{exceptional hyperelliptic}. The fourth column of Table \ref{tab:results} lists the number of exceptional hyperelliptic fields $K$ in the database whose Galois group is isomorphic to the group specified in the first column.

Finally, the fifth column of Table \ref{tab:results} indicates the number of fields with the specified Galois group for which there heuristically exists both a hyperelliptic and a non-hyperelliptic curve whose Jacobian has CM by $\Z_K$. For simplicity, we call such a sextic CM field $K$ \defi{mixed}. Note that all mixed fields are necessarily exceptional hyperelliptic, since by the previous paragraph no non-hyperelliptic CM curves can exist for $K$ if it contains $\Q (i)$.

\begin{table}[h]
  \begin{tabular}{|l|r|r|r|r|}
    \hline
    Galois group & $\# K$ & $\text{$\#$ hyp.\ $K$}$ & $\text{$\#$ exc.\ hyp.\ $K$}$ & $\text{$\#$ mixed $K$}$ \\
    \hline
    $C_6$                &  10,067 &   348 &  2 &  0 \\
    $D_6$                &  32,544 & 3,057 &  0 &  0 \\
    $C_2^3 \rtimes C_3$  &  10,159 &     0 &  0 &  0 \\
    $C_2^3 \rtimes S_3$  & 494,386 &    17 & 17 & 14 \\
    \hline
    Total                & 547,156 & 3,422 & 19 & 14 \\
    \hline
  \end{tabular}
  \caption{CM fields in the LMFDB}
  \label{tab:results}
\end{table}

\subsection{Invariants}\label{sec:invs}

Let $\tau \in M_{g,g} (\C)$ be a small period matrix. This section briefly reviews what is known on calculating and algebraizing the invariants of the curve $X$ associated to $\tau$, as well as verifying the correctness of the resulting curve.

Algorithm \ref{alg:algo4} shows that we can compute an approximation to $\tau$ to a given high precision, as all that we need to do is to determine the image of a basis of a (minimized) representative $\aa$ under the given CM type $\Phi$. What is considerably more complicated is to compute the even theta-null values associated to $\tau$. Here it is in general essential to use the more sophisticated algorithms by Labrande \cite{labrande} to keep the running time within reasonable bounds. While the available implementation of this algorithm does not always work properly, we still managed to get by in the cases that interested us, either by using the naive method from \cite{labrande} to lower precision or by determining the even theta-null values for only a single element of a given Galois orbit and conjugating afterwards in the next algebraization step.

Given the even theta-null values, we can determine a model of $X$ over $\C$ to the given precision, either by using the Rosenhain invariants as in \cite{BILV} or by using the Weber model from \cite{KLRRSS}. We can then compute a normalized weighted representative $I$ of the corresponding invariants (using the Shioda invariants in the hyperelliptic case and the Dixmier--Ohno invariants in the non-hyperelliptic case). The field of moduli of $X$ then coincides with the field generated by the entries of $I$.

\medskip
\emph{Algebraization.} It remains to algebraize the invariants $I$. A first possible method is the usual application of the LLL algorithm to determine putative minimal polynomials of the entries of $I$ over $\Q$ and thus to obtain $I$ as elements of a number field. One corresponding implementation is \texttt{NumberFieldExtra} in \cite{cms-git}. A second method is to symmetrize and use class polynomials, as in \cite{DI, engethome}. Both of these methods became prohibitive in the cases that we considered because of the large heights of the algebraic numbers that were involved. Indeed, one of the mixed fields, defined by the polynomial $x^6 - 2 x^5 + 11 x^4 + 42 x^3 - 11 x^2 + 340 x + 950$, gives rise to a tuple of normalized Dixmier--Ohno invariants whose first non-trivial entry $I_6$ has height $\approx 2.94 \cdot 10^{431}$, with $I_{27}$ having a height that is even larger by an exponential factor of about $27/6$. Another reason for us not to use the class polynomial method from \cite{engethome} is that this would necessitate later factorization to determine the Galois orbits, which is superfluous when algebraizing the individual $I$ directly.

Instead we exploited the fact that that the Shimura reciprocity law implies that the entries of $I$ are the complex embeddings of elements of Hilbert class field $H$ of the reflex field $K^r$. This replaces the problem of determining minimal polynomials to the more tractable one of trying to algebraize the elements of $I$ in $H$ or its subfields, which also reduces to an application of LLL (for example in the form of the routine \texttt{AlgebraizeElementsExtra} in \cite{cms-git}). In the aforementioned complicated case we needed $20,000$ digits of precision for our algebraization, but usually around $3,000$ digits were enough. Incidentally, note that while the reflex $K^r$ itself can be costly to determine via the usual Galois theory, since the closure $L$ becomes quite large, it can still be quickly recovered numerically as a subfield of $\C$, namely by applying the methods from the previous paragraph.

\medskip
\emph{Verification.} Once we have algebraized the elements of $I$, we have applied heuristic numerical methods twice, both in the determination of $I$ itself and in the algebraization of its elements. One may well ask why one should trust the algebraic invariant values thus obtained to be correct. Here are several reasons:
\begin{enumerate}
  \item For all algebraizations $I$ that we found, the resulting invariants satisfy the known algebraic dependencies between the Shioda invariants (which can be found in \cite{lrs-hyper}) or the Dixmier--Ohno invariants (which can be found in \cite{lrs-plane}). There is no reason whatsoever for this to hold in the case of incorrect or badly algebraized $I$.
  \item Reducing the values of $I$ modulo various large primes, one can apply the reconstruction algorithms from \cite{lr} or \cite{lrs} and then compute Weil polynomials to check that the resulting curves indeed have CM by an order in $K$ for all these primes.
  \item Conversely, one can directly calculate the set $S$ of primes of bad reduction from $I$ by using the criteria in~\cite{LLLR}, or compute a bound on primes of bad reduction and check $S$ against the set of primes found by running the algorithm in~\cite{win-new} up to the bound.
  \item In principle one can verify all results obtained over $\Qbar$ by using \cite{cmsv-endos}. That said, these algorithms still need substantial speedups for these verifications to be feasible for plane quartic curves over number fields.
\end{enumerate}
This is why we do not harbor any doubts about our results being correct, even though they are by no means mathematically rigorous yet.

\subsection{The mixed cases}

Table \ref{tab:mixed} describes the results for the $17$ fields $K$ from Table \ref{tab:results} with Galois group $C_2^3 \rtimes S_3$ that are exceptional. Note that there are also $2$ exceptional hyperelliptic fields with Galois group $C_6$, but these were already considered in \cite{resnt}: Corresponding polynomials are given by $x^6 - x^5 + x^4 - x^3 + x^2 - x + 1$ and $x^6 - 14 x^3 + 63 x^2 + 168 x + 161$.

The first column of Table \ref{tab:mixed} gives the polynomial defining the CM field $K$; this column is sorted by the absolute discriminant of the ring of integers $\Z_K$. The second column describes the length of the various hyperelliptic Galois orbits under conjugation by $\Gal (\Qbar \ext \Q)$; for example, an entry $4^2 8^1$ stands for $2$ Galois orbits of length $4$ along with single Galois orbit of length $8$. Similarly, the third column describes the length of the non-hyperelliptic Galois orbits under $\Gal (\Qbar \ext \Q)$. An empty entry means that there does not exist such a curve for the field $K$. Note that Corollary \ref{cor:fom} shows why the length of the Galois orbits in the table are all a multiple of $4$. The final column describes the quotient $\Cc_K / \im (\Nc_{\Phi})$ of the Shimura class group by the image of the reflex type norm. Note that this independent of the chosen primitive CM type $\Phi$ because of Proposition \ref{prop:conjtype} and Corollary \ref{cor:trans}.

The invariants obtained for the fields in Table \ref{tab:mixed} are available at \cite{dis-github}. As mentioned above, they are occasionally on the gargantuan side.

\begin{table}[h]\label{table:orbits}
  \begin{tabular}{|l|r|r|c|}
    \hline
    CM field & hyp.\ orbits & non-hyp.\ orbits & $\Cc_K / \im (\Nc_{\Phi})$ \\
    \hline
    $x^6 + 10 x^4 + 21 x^2 + 4$ & $4^1$ & $4^1$ & $\Z / 2 \Z$ \\
    $x^6 - 3 x^5 + 14 x^4 - 23 x^3 + 28 x^2 - 17 x + 4$ & $4^1$ & $4^1$ & $\Z / 2 \Z$  \\
    $x^6 - 2 x^5 + 12 x^4 - 31 x^3 + 59 x^2 - 117 x + 121$ & $4^1$ & $4^1 8^1$ & $\Z / 4 \Z$ \\
    $x^6 + 14 x^4 + 43 x^2 + 36$ & $4^1$ & & 1 \\
    $x^6 - 3 x^5 + 9 x^4 + 4 x^3 + 12 x^2 + 84 x + 236$ & $4^1$ & $4^1 8^1$ & $\Z / 4 \Z$ \\
    $x^6 - 2 x^5 + x^4 - 4 x^3 + 5 x^2 - 50 x + 125$ & $4^1$ & $4^3$ & $(\Z / 2 \Z)^2$ \\
    $x^6 + 29 x^4 + 246 x^2 + 512$ & $4^1$ & & 1 \\
    $x^6 - 3 x^5 + 10 x^4 + 8 x^3 + x^2 + 90 x + 236$ & $4^1$ & $4^1$ & $\Z / 2 \Z$ \\
    $x^6 + 21 x^4 + 60 x^2 + 4$ & $4^1$ & $4^1$ & $\Z / 2 \Z$ \\
    $x^6 + 30 x^4 + 169 x^2 + 200$ & $4^1$ & $4^1$ & $\Z / 2 \Z$ \\
    $x^6 + 23 x^4 + 112 x^2 + 36$ & $4^1$ & & 1 \\
    $x^6 - 2 x^5 + 12 x^4 - 44 x^3 + 242 x^2 - 672 x + 1224$ & $12^1$ & $12^3$ & $(\Z / 2 \Z)^2$ \\
    $x^6 + 26 x^4 + 177 x^2 + 128$ & $4^1$ & $4^1$ & $\Z / 2 \Z$ \\
    $x^6 + 29 x^4 + 226 x^2 + 252$ & $4^1$ & $4^1 8^1$ & $\Z / 4 \Z$ \\
    $x^6 - 2 x^5 - 7 x^4 + 45 x^3 - 63 x^2 - 162 x + 729$ & $4^1$ & $4^1$ & $\Z / 2 \Z$ \\
    $x^6 - 2 x^5 + 11 x^4 + 42 x^3 - 11 x^2 + 340 x + 950$ & $8^1$ & $8^1 16^1$ & $\Z / 4 \Z$ \\
    $x^6 - 3 x^5 + 29 x^4 - 53 x^3 + 200 x^2 - 174 x + 71$ & $4^1$ & $4^1$ & $\Z / 2 \Z$ \\
    \hline
  \end{tabular}
  \caption{Generic hyperelliptic and mixed fields with Galois group $C_2^3 \rtimes S_3$ and the lengths of the corresponding Galois orbits}
  \label{tab:mixed}
\end{table}

\section{Explicit defining equations}\label{sec:equations}

In this section we further consider the mixed CM field $K$ defined by $x^6 + 10 x^4 + 21 x^2 + 4$, which corresponds to the first entry of Table \ref{tab:mixed}. Our goal is to indicate how to obtain the (heuristic) explicit defining equations from Main Result 3. The actual calculations are performed in \cite{dis-github}; here we briefly explain the ideas that underlie them.

There are two Galois orbits in this case, one containing $4$ hyperelliptic curves, and one containing $4$ non-hyperelliptic curves. Moreover, Corollary \ref{cor:trans} shows that once we fix a CM type $\Phi$, which we do throughout this section, there is exactly one corresponding hyperelliptic curve $X$ and one non-hyperelliptic curve $Y$. We start by finding an equation for $X$.

\subsection{Hyperelliptic simplification}\label{sec:hypsimp}

As in Section \ref{sec:invs}, we determine a normalized tuple $S$ of Shioda invariants corresponding to the curve $X$, which is defined over the quartic field $L$ defined by the polynomial $x^4 - 5 x^2 - 2 x + 1$. The field $L$ is in fact the totally real subfield of the reflex field $K^r$ of $K$.

One can try to apply the generic reconstruction algorithms in genus $3$ that are available in \Magma, but this turns out not to be optimal, as the resulting hyperelliptic curve is returned over a random quadratic extension of $L$ with large defining coefficients. Instead, we directly construct the Mestre conic and quartic $Q$ and $H$ over $K$ from the invariants $S$, as in \cite{lr}, and then check whether the conic $Q$ admits a rational point. This turns out to be the case. Choosing a parametrization $\P^1 \to Q$ over $K$ and pulling back the divisor $Q \cap H$ on $Q$, we obtain a degree-8 divisor on $\P^1$ that corresponds to a monic octic polynomial $f$ with the property that
\begin{equation}\label{eq:X}
  X : y^2 = f
\end{equation}
is a curve with CM by $\Z_K$. This is still far from satisfactory, however, as the coefficients of $f$ are extremely large, namely of height up to $4.92 \cdot 10^{1126}$. We show how to obtain a simpler equation. Our approach is essentially ad hoc; while there are minimization and reduction algorithms in \Magma\ over the rationals due to Cremona--Stoll \cite{cremona-stoll}, and over real quadratic number fields due to Bouyer--Streng \cite{bouyer-streng}, we do not find ourselves in one of these cases, so that we are forced to use other methods.

\medskip
The octic polynomial $f$ factors as
\begin{equation}
  f = f_1 f_2 f_3,
\end{equation}
where $f_1$ and $f_2$ are quadratic, both defined over a pleasant quadratic extension $M$ of $L$ with defining polynomial $x^8 - 4 x^7 + 10 x^5 + 7 x^4 - 10 x^3 - 18 x^2 - 6 x + 1$ over the rationals. (That this extension is so agreeable is of course no surprise; the extended version of the Main Theorem of complex multiplication, applied to the $2$-torsion of $\Jac (X)$, shows that we should expect it to be related to the Hilbert class field of $L$ ramifying at its single even prime.)

We now consider $f$ over the quadratic extension $M$ of $L$, over which field we will construct a simpler polynomial defining the same hyperelliptic curve, which we will then descend back to $L$. To start our simplification over $M$, we apply a Möbius transformation in the $x$-coordinate that sends the roots of $f_1$ to $0$ and $\infty$ and one of the roots of $f_2$ to $1$. This maps the divisor defined by the octic polynomial $f$ to that defined by a \emph{septic} polynomial $g$ that additionally satisfies $g (0) = g (1) = 0$. We normalize $g$ in such a way that the coefficient of $x^4$ equals $1$, for reasons that will become clear, so that
\begin{equation}\label{eq:g}
  g = c_7 x^7 + c_6 x^6 + c_5 x^5 + x^4 + c_3 x^3 + c_2 x^2 + c_1 x .
\end{equation}
At this point the maximal height of the coefficients of $g$ is $5.51 \cdot 10^{17}$, which is already quite a bit smaller than $4.92 \cdot 10^{1126}$. Now inspecting the norms of the coefficients $c_i$ shows that we have
\begin{equation}
  (c_5) = \pp_2^{-4} (\sigma (c_3)),
\end{equation}
where $\pp_2$ is the unique ideal of $\Z_L$ above $2$ and where $\sigma$ is the involution that generates $\Gal (M \ext L)$. Following a hunch, we scale $x$ by $\alpha^2$, where $\alpha$ generates $\pp_2$. Transforming $g$ accordingly, we obtain an equality of ideals
\begin{equation}\label{eq:switchid}
  (c_i) = (\sigma (c_{8 - i}))
\end{equation}
for all $i$ between $1$ and $4$.

Our goal is to make Equation \eqref{eq:switchid} hold on the level of elements, and not merely between ideals. To achieve this, we consider the unit $u = c_5 / \sigma (c_3) \in \Z_M^*$. Consider the polynomial $h$ obtained from $g$ by scaling $x$ by $v x$, where $v \in \Z_M^*$ is another unit, and normalizing the coefficient of $x^4$ to equal $1$. Then for the coefficients $d_i$ of $h$ we have $d_5 = v c_5$ and $d_3 = v^{-1} c_3$. The equality of elements $d_5 = \sigma (d_3)$ that we are looking for can be rewritten as
\begin{equation}
  v u \sigma (c_3) = v c_5 = d_5 = \sigma (d_3) = \sigma (v^{-1}) \sigma (c_3)
\end{equation}
This is the case if and only if
\begin{equation}\label{eq:uniteq}
  u = v \sigma (v) .
\end{equation}
Since $\Z_M^* \cong \Z / 2 \Z \times \Z^5$ and $v$ satisfies this equality if and only of $-v$ does, the equality \eqref{eq:uniteq} reduces to integral linear algebra once the representation of $\sigma$ in terms of a given basis of $\Z_K^*$ is known. Performing the corresponding computation shows that we can indeed find a $v$ with the requested properties. Scaling $x$ accordingly, we find a polynomial $g$ as in \eqref{eq:g} such that
\begin{equation}\label{eq:switchelt}
  c_i = \sigma (c_{8 - i})
\end{equation}
is satisfied for all $i = 1, \ldots, 4$.

At this point, our manipulations have lead to a polynomial $g$ with coefficients of maximal height $8.64 \cdot 10^{16}$. We can still do a bit better by further scaling $x$ by appropriate units $v$ satisfying $v \sigma (v) = 1$. This does not affect the property \eqref{eq:switchelt}. Our goal is to make $c_5$ as small as possible as an element of the Minkowski lattice up to shifts by units of the indicated type. This reduces to a closest vector problem as in Remark \ref{rem:algo1}, an approximation for which is quickly found by means of the usual techniques. Applying the corresponding scaling again yields a polynomial $g$ with coefficients of height $1.11 \cdot 10^{16}$.

\medskip
It now remains to descend our polynomial $g$ with coefficients in $M$ to the original field $L$. For this, let $\sigma$ be the involution that generates the Galois group $\Gal (M \ext L)$, and let $B \in \GL_2 (M)$ be such that
\begin{equation}
  \text{$\sigma (B) = A B$, \qquad where $A = \begin{pmatrix} 0 & 1 \\ 1 & 0 \end{pmatrix}$} .
\end{equation}
For example, we can take
\begin{equation}
  B = \begin{pmatrix}
    1 &         \alpha \\
    1 & \sigma (\alpha)
  \end{pmatrix}
\end{equation}
where $\alpha \in M$ is such that $M = L \oplus L \alpha$. The matrix $B^{-1}$ induces a Möbius transformation of the projective line.

\begin{proposition}
  Let $D = (g)_0 \cup \left\{ \infty \right\} \subset \P^1$, where $(g)_0$ is the divisor of zeros of the polynomial $g$, and let $E_0 = B^{-1} (D)$.
  \begin{enumerate}
    \item The divisor $E_0 \subset \P^1$ is defined over $L$.
    \item Let $f_0$ be a polynomial whose divisor of zeros is given by $E_0$. Then the hyperelliptic curve $X_0 : y^2 = f_0$ over $L$ is isomorphic over $\Qbar$ to the original curve $X$ in \eqref{eq:X}.
  \end{enumerate}
\end{proposition}

\begin{proof}
  (i): The divisor $E_0$ is defined over $M$, as $B$ and $D$ are. Moreover, we have
  \begin{equation}
    \sigma (E_0) = \sigma (B^{-1} D) = \sigma (B)^{-1} \sigma (D) = B^{-1} A^{-1} A D = B^{-1} D = E_0
  \end{equation}
  This Galois invariance implies our claim.

  (ii) This follows from (i) because two hyperelliptic curves are $\Qbar$-isomorphic if (and only if) the corresponding branch loci are related by a Möbius transformation.
\end{proof}

Alternatively, $f_0$ is the numerator of the transform of $g$ by $B^{-1}$. This turns out to be still of reasonable size when $\alpha$ is. Replacing $X$ be $X_0$, we have achieved our aim of simplifying $X$. The result is the equation for $X$ in Main Result 3. The discriminant of the corresponding hyperelliptic polynomial equals $\pp_4^{120} \pp_7^{12}$, where $\pp_4$ (resp.\ $\pp_7$) is an ideal of norm $4$ (resp.\ $7$).

\subsection{A plane quartic equation}

It remains to construct a plane quartic model for the non-hyperelliptic curve $Y$ from the knowledge of its Dixmier--Ohno invariants $I$. The direct methods from \cite{lrs} gives a ternary quartic with coefficients whose size is beyond hopeless. Methods to obtain defining equations of smaller size were sketched in \cite[\S 3]{KLRRSS}, using methods due to Elsenhans and Stoll \cite{elsenhans, stoll}, yet like the methods of Cremona--Stoll in Section \ref{sec:hypsimp}, these are specific to the base field $\Q$, and therefore of no use in the current situation.

Fortunately, now that we have found the equation for the hyperelliptic curve $X$ in Main Result 3, determining the equation for the non-hyperelliptic curve $Y$ becomes tractable. To see this, let $P_X \in M_{3,6} (\C)$ be a big period matrix of $X$ with respect to the canonical basis of differentials $\left\{ dx / y, x dx / y, x^2 dx / y \right\}$ corresponding to the equation \eqref{eq:Xsmall}, and let $P_Y$ be the large period matrix of the Weber model $Y : F (x, y, z) = 0$ over $\C$ for $Y$ obtained in the course of using Algorithm \ref{alg:algo4}. This matrix, and all other big period matrices that follow, should be taken with respect to the canonical basis of differentials $\left( x dx (\partial F / \partial y)^{-1}, y dx (\partial F / \partial y)^{-1}, dx (\partial F / \partial y)^{-1} \right)$.

\begin{proposition}
  There exist matrices $T \in M_{3,3} (\C)$ and $R \in M_{6,6} (\Z)$ such that $R$ has determinant $2$ and
  \begin{equation}
    T P_Y = P_X R .
  \end{equation}
  Moreover, the pair $(T, R)$ is uniquely determined up to a minus sign.
\end{proposition}

\begin{proof}
  This is a direct consequence of the fact that $X$ and $Y$ are related by an $\aa$-transformation with $N_{K | \Q} (\aa) = 2$. In turn, this statement follows from the fact (see Table \ref{tab:mixed}) that $\Cc_K / \im (\Nc) \cong \Z / 2 \Z$, and that if we factor $(2) = \aa^4 \bb^2$ in $\Z_K$, with $N_{K | \Q} (\aa) = N_{K | \Q} (\bb) = 2$, the ideal $\aa$ represents the non-trivial class in this quotient, which therefore induces an isogeny between the two distinct ppavs with CM by $K$ of a fixed type $\Phi$. The uniqueness claim follows from the fact that $\aa$ is the only ideal of norm $2$ that gives rise to a non-trivial class in $\Cc_K / \im (\Nc)$.
\end{proof}

In what follows, given a matrix $T \in M_{3,3} (\C)$ and a ternary quartic form $F \in \C [x,y,z]$, we denote the transformation of $F$ under the natural right action of $T$ by $F \cdot T$.

\begin{proposition}
  Let $F$ be the ternary quartic form associated to the Weber model whose big period matrix is $P_Y$, and $F_0$ be a multiple of $F \cdot T^{-1}$ that is normalized in such a way that one of its coefficients is in $L$. Then $Y_0 : F_0 (x, y, z) = 0$ is a model of $Y$ over $L$.
\end{proposition}

\begin{proof}
  We know that $Y$ has field of moduli equal to $L$. Now since the torsion subgroup of $\Z_K^*$ is generated by $\langle -1 \rangle$, the automorphism group $\Aut (Y)$ is trivial, since $\Aut (Y) = \Aut (\Jac (Y)) / \langle -1 \rangle$ for plane quartic curves $Y$. Therefore there exists a plane quartic curve $Z \subset \P^2$ defined over $L$ that is isomorphic to $Y$. Let $G$ be a corresponding form, and let $P_Z$ be a corresponding period matrix. The same argument as above shows that there exists a matrix $U \in M_{3,3} (\C)$ such that
  \begin{equation}
    U P_Z = P_X R .
  \end{equation}
  Because both $X$ and $Z$ are defined over $L$, the uniqueness of $R$ up to sign implies that $U \in M_3 (\Qbar)$ and $\sigma (U) = \pm U$ for all $\sigma \in \Gal (\Qbar \ext \Q)$. Now let $G_0 = G \cdot U^{-1}$, normalized in such a way that one of its coefficients is in $L$. Since $\sigma (G \cdot U^{-1}) = \sigma (G) \cdot \sigma (U^{-1}) = G \cdot \pm U^{-1}$, we have that the class of $G \cdot U$ up to scalar is Galois stable. Therefore $G_0$ is defined over $L$, and its big period matrix is a scalar multiple of $U P_Z = P_X R$. On the other hand, the ternary quartic $F \cdot T^{-1}$ also has a big period matrix that is a scalar multiple of $T P_Y = P_X R$. Therefore $F_0$ and $G_0$ coincide up to a scalar, and because of our normalization $F_0$ has coefficients in $L$ as well.
\end{proof}

An algebraization in the field $L$ using LLL shows that we can indeed recover the coefficients of the ternary quartic form $F_0$ defining $Y_0$ over $K$. Tweaking its size by scaling $x, y, z$ by units (similar to the closest vector considerations in Section \ref{sec:hypsimp}) makes the equation of $Y_0$ somewhat smaller still. Replacing $Y$ by $Y_0$ gives the equation for $Y$ in Main Result 3. Its discriminant factors as $\pp_4^{312} \pp_7^{36} \pp_{19}^{14} \pp_{277}^{14} \pp_{1753}^{14}$, where as before subscripts indicate norms.

\begin{remark}
  We emphasize once more that the equations obtained in this section have not yet been verified by the methods from \cite{cmsv-endos} because of the considerable effort required to run these algorithms over large number fields.
\end{remark}

\section{Around the André--Oort conjecture}\label{sec:andreoort}

\subsection{General considerations}

In this section, we review a certain number of results around the André--Oort conjecture. The André--Oort conjecture was formulated in the general context of Shimura varieties and their special points. A proof of this conjecture under the assumption of the generalized Riemann hypothesis for CM fields has been given by Klingler and Yafaev \cite{Klingler}. For an extensive survey on Shimura varieties and a general statement of the conjecture, the reader is referred to \cite{MoonenOort}.

Although our focus is on genus 3, we start by stating facts that hold for every $g\geq 1$. We denote by $\mathcal{A}_g$ the moduli space of ppavs of dimension $g$ over $\C$ and by $\mathcal{M}_g$ the moduli space of smooth genus $g$ curves defined over $\C$. Recall that the Torelli morphism
\begin{equation}
  j: \mathcal{M}_g\rightarrow \mathcal{A}_g
\end{equation}
associates to every curve its principally polarized Jacobian. We denote by $\mathcal{T}_g$ the closed Torelli locus, i.e., $\mathcal{T}_g=\overline{j(\mathcal{M}_g)}$.

As a complex variety, $\mathcal{A}_g=\text{Sp}_{2g}(\Z)\backslash \mathcal{H}_g$ is a Shimura variety whose special points are exactly the CM points. Recently, Tsimerman \cite{Tsimerman} proved a result showing the existence of a lower bound on the size of the Galois orbits of CM points in $\mathcal{A}_g$. This result, combined with joint work with Pila~\cite{PilaTsimerman}, allowed him to complete a proof of the André-Oort conjecture for $\mathcal{A}_g$ without the generalized Riemann hypothesis assumption.

\begin{theorem}[André--Oort conjecture \cite{Tsimerman}]
  Let $\Gamma$ be a set of CM points in $\mathcal{A}_g$. Then the Zariski closure of $\Gamma$ is a finite union of Shimura subvarieties.
\end{theorem}


Among the Shimura subvarieties of $\mathcal{A}_g$, a well known example is that of the Hilbert modular variety, whose points are polarized abelian varieties whose endomorphism ring contains the ring of integers of a totally real field of genus $g$. Hilbert modular varieties play an important role when studying the number of CM points in $\mathcal{T}_g$.

Indeed, let us turn our attention to the case of CM fields with Galois group isomorphic to $C_2^g\rtimes S_g$. Chai and Oort call these fields and their corresponding CM points sufficiently general (see \cite[(2.13)]{Chai} for a justification of this definition). We will use the following result given in \cite{Chai}.

\begin{lemma}\label{Chai}
  Let $Y$ be an irreducible Shimura subvariety of $\mathcal{A}_g$ of positive dimension. Assume that $Y\neq \mathcal{A}_g$ and that $Y$ contains a sufficiently general CM point $y$ in $\mathcal{A}_g$. Then $Y$ is a Hilbert modular variety attached to the totally real subfield of degree $g$ over $\Q$ contained in the CM field attached to $y$.
\end{lemma}

This lemma allowed the authors of \cite{Chai} to establish the following result for genus $g>3$.
\begin{theorem}
  Assume the André--Oort conjecture to be true. Then for every $g>3$ the number of sufficiently general CM points in $\mathcal{T}_g$ is finite.
\end{theorem}

When $g=3$, the closed Torelli locus $\mathcal{T}_3$ coincides with $\mathcal{A}_3$, but we believe that a similar argument can be adapted to genus 3, as soon as we restrict to the hyperelliptic locus. Indeed, let us denote by $\mathcal{M}_3^{\hyp}$ the image of the subspace of hyperelliptic curves inside the Torelli locus. Then $\mathcal{M}_3^{\hyp}$ contains infinitely many hyperelliptic curves with CM, since all genus 3 curves with CM by a field containing $\Q(i)$ are hyperelliptic. This is certainly in accordance with the André--Oort conjecture, since the Shimura surface parametrizing points whose endomorphism ring contains $\sqrt{-1}$ is contained in $\mathcal{M}_3^{\hyp}$.

Assume now that $\mathcal{M}_3^{\hyp}$ contains infinitely many sufficiently general CM points. Then by the André--Oort conjecture and Lemma \ref{Chai}, it contains a Hilbert modular variety attached to a totally real field of degree 3. Recall that among the exceptional hyperelliptic fields listed in Table \ref{tab:mixed}, 14 are mixed, i.e., they allow both a hyperelliptic and non-hyperelliptic curve. This quickly disproves the fact that the Hilbert modular variety corresponding to the real multiplication subfield of each of these fields could be contained in the hyperelliptic locus. For the remaining 3 exceptional hyperelliptic fields listed in the Table, we cannot reach a similar conclusion for the corresponding real multiplication subfields and their Hilbert modular varieties. One way to tackle the question experimentally would be to adapt our implementation to compute points with CM by non-maximal orders, which contain the maximal real multiplication order in these fields. Once the period matrices of these points are determined, it would suffice to use Algorithm \ref{alg:algo4} to check heuristically that some of the corresponding curves are non-hyperelliptic.

As stated in the introduction, we do not have enough evidence to support the claim that the list of exceptional hyperelliptic CM fields mentioned in Main Result 1 and 2 is complete and we certainly do not claim that. However, the considerations above support the conjecture that the full list of exceptional hyperelliptic CM fields should be finite.

\subsection{Cryptographic implications}\label{sec:crypto}

Let us now turn our attention to applications in cryptography. The Discrete Logarithm Problem (DLP) in Jacobians of hyperelliptic curves defined over a finite field $\F_q$ (with $q=p^d$ and $p$ a prime) can be solved in $\tilde{O}(q^{4/3})$, using the index calculus algorithm of Gaudry, Thériault and Diem \cite{Gaudry}. In contrast, Jacobians of non-hyperelliptic curves of genus 3 are amenable to Diem’s index calculus algorithm, which requires only $\tilde{O}(q)$ group operations to solve the DLP \cite{Diem}. As a consequence, an efficient way of attacking DLP on a genus 3 hyperelliptic Jacobian is by reducing it to a DLP on a non-hyperelliptic Jacobian via an explicit isogeny. Assuming that the kernel of the isogeny will intersect trivially with the subgroup of cryptographic interest, we derive a $\tilde{O}(q)$ time attack on the hyperelliptic Jacobian (see \cite{Smith}). So an interesting question is how to find such isogenies.\\

\noindent
\textit{Idea of the attack.} To tackle this question, let us consider $A$ an ordinary ppav defined over $\F_q$ isomorphic to a hyperelliptic Jacobian. The theory of canonical lifts of Serre and Tate allows us to lift $A$ to an ordinary ppav $\tilde{A}$ defined over $W(\F_q)$, the ring of Witt vectors of $\F_q$, such that $\End (A) \simeq \End (\tilde{A})$ and $A\rightarrow \tilde{A}$ is functorial (see \cite{Bowdoin}). After fixing an embedding $W(\bar{\F}_q)\hookrightarrow \C$, we may assume that $\tilde{A}$ is a ppav defined over $\C$ with CM by the maximal ring of integers of $K$ and CM type $\Phi$. As suggested by our Main Results 1 and 2, hyperelliptic Jacobians with CM are rare, hence most of the times we expect $\tilde{A}$ to be a non-hyperelliptic Jacobian with hyperelliptic reduction mod $p$. We now consider the following graph: the vertices are absolutely simple 3-dimensional ppav defined over $\C$ with CM by the maximal order of $K$ and the edges are isogenies between ppavs. In the literature, this is known as the \textit{horizontal isogeny graph} (see for instance  \cite{Jetchev}). Moreover, by \cite[Ch. III, Sec. 11, Prop. 13]{Shimura}, the isogenies in this graph will reduce to isogenies defined over $\F_q$ of equal degree.

In this graph, our goal is to find an isogeny from $\tilde{A}$ to another ppav, which has good quartic reduction at $p$. The problem is not trivial, since the number of vertices in this graph is $O(\#\Cc_K)$, hence it grows exponentially with the size of the class group of $K$. If we construct an isogeny to a ppav on the Galois orbit of $\tilde{A}$ as in Theorem \ref{thm:srecip}, then the target variety will also have hyperelliptic reduction at $p$.

Consequently, we will choose an isogeny $\tilde{I}$ corresponding to a non-trivial element in $\Cc_K/\im (\Nc_{\Phi})$ (preferably one which allows an ideal representative of smallest possible norm). We denote by $\tilde{B}$ the target ppav obtained in this way and by $B$ its reduction modulo $p$. Heuristically, both $\tilde{B}$ and $B$ are isomorphic to non-hyperelliptic Jacobians. To support this heuristic, we computed all primes of hyperelliptic reduction for all non-hyperelliptic orbits for a given CM field.

\begin{example}
  As an example, we revisit the case of the CM field of equation $x^6 - 2x^5 + x^4 - 4x^3 + 5 x^2 - 50 x + 125$, which is the sixth entry in Table \ref{tab:mixed}. Recall that for this field there is one hyperelliptic orbit of length 4 and three non-hyperelliptic orbits under conjugation by $\Gal(\bar{\Q}/\Q)$. The Dixmier-Ohno invariants of plane quartics with CM by this field are defined over a degree 4 extension field of $\Q$ of equation $x^4-17x^3-24x^2+7$. We computed invariants for one curve on each of the non-hyperelliptic orbits (see \cite{dis-github} for the numerical values). With these in hand, we computed the primes of hyperelliptic reduction for these CM points, using the criterion in \cite[Theorem 1.10]{LLLR}. We list the results in Table \ref{tab:hyperellred}, where as before the subscripts denote the norms of the ideals. We can see that the lists of primes of hyperelliptic reduction for different orbits are almost disjoint (only $\pp_{29}$ appears in two of these lists).
\end{example}

\begin{table}[h]
  \begin{tabular}{|c|c|}
    \hline
    Orbit & Prime ideals of hyperelliptic reduction\\
    \hline
    $1$ & $ \pp_{29}, \pp_{151},\p_{331},\pp_{15937}, \pp_{2986259}$\\
    \hline
    $2$ & $\pp_{29},\pp_{53}, \pp_{409}, \pp_{2251},\pp_{27509}, \pp_{37423}, \pp_{154757110537}$\\
    \hline
    $3$ & $\pp_{71}, \pp_{827},\pp_{2207}, \pp_{3181},\pp_{6133}$ \\
    \hline
  \end{tabular}
  \caption{Hyperelliptic reduction for non-hyperelliptic curves with CM by the field with defining polynomial $x^6 - 2x^5 + x^4 - 4x^3 + 5 x^2 - 50 x + 125$}
  \label{tab:hyperellred}
\end{table}

\bibliographystyle{abbrv}
\bibliography{references}

\end{document}